\documentclass[11pt,a4paper,twoside]{article}

\usepackage[T1]{fontenc}

\usepackage{amsmath,amssymb,latexsym,amsthm,varioref}

\usepackage{graphicx,subfig,xcolor}
\labelformat{figure}{Figure~#1}

\renewcommand{\geq}{\ensuremath{\geqslant}}
\renewcommand{\leq}{\ensuremath{\leqslant}}
\newcommand{\N}{\ensuremath{\mathbb{N}}}
\newcommand{\Z}{\ensuremath{\mathbb{Z}}}
\newcommand{\R}{\ensuremath{\mathbb{R}}}
\newcommand{\1}{\ensuremath{\mathbf{1}}}
\renewcommand{\L}{\ensuremath{\mathbb{L}}}
\newcommand{\E}{\ensuremath{\mathbb{E}}}
\renewcommand{\P}{\ensuremath{\mathbb{P}}}
\def\var{\mathrm{Var}}
\newcommand{\f}{\ensuremath{\varphi_{\lambda}}}

\numberwithin{equation}{section}

\theoremstyle{plain}  
\newtheorem{thm}{Theorem}
\labelformat{thm}{Theorem~#1}
\newtheorem{cor}{Corollary}[section]
\labelformat{cor}{Corollary~#1}
\newtheorem{lem}{Lemma}[section]
\labelformat{lem}{Lemma~#1}
\theoremstyle{definition}  
\theoremstyle{remark}  

\title{\textbf{Wavelet thresholding estimation in a Poissonian interactions model with application to genomic data}}
\author{Laure \textsc{Sansonnet}
        \medskip \\
        \emph{Laboratoire de Mathématiques, CNRS UMR 8628} \\
        \emph{Université Paris-Sud 11} \\
        \emph{15 rue Georges Clémenceau, 91405 Orsay Cedex, France}}
\date{2011}
\begin{document}

\maketitle

\begin{center}
\begin{minipage}{0.75\textwidth}
\noindent \textbf{Abstract:} This paper deals with the study of dependencies between two given events modeled by point processes. In particular, we focus on the context of DNA to detect favored or avoided distances between two given motifs along a genome suggesting possible interactions at a molecular level.
For this, we naturally introduce a so-called reproduction function $h$ that allows to quantify the favored positions of the motifs and which is considered as the intensity of a Poisson process.
Our first interest is the estimation of this function $h$ assumed to be well localized.
The estimator $\tilde{h}$ based on random thresholds achieves an oracle inequality.
Then, minimax properties of $\tilde{h}$ on Besov balls $\mathcal{B}^{s}_{2,\infty}(R)$ are established.
Some simulations are provided, allowing the calibration of tuning parameters from a numerical point of view and proving the good practical behavior of our procedure.
Finally, our method is applied to the analysis of the influence between gene occurrences along the \emph{E.\,coli} genome and occurrences of a motif known to be part of the major promoter sites for this bacterium.

\medskip

\noindent \textbf{Keywords:} Adaptive estimation, interactions model, oracle inequalities, Poisson process, thresholding rule, $U$-statistics, wavelets.

\medskip

\noindent \textbf{MSC2010:} Primary 60G55, 62G05; secondary 62G20, 62G30.
\end{minipage}
\end{center}

\section{Introduction}

The goal of the present paper is to study the dependence between two given events modeled by point processes. We propose a general statistical approach to analyze any type of interaction, for instance, interactions between neurons in neurosciences or the comprehension of bankruptcies by contagion in economics. In particular, we focus on a model to study favored or avoided distances between patterns on a strand of DNA, which is an important task in genomics.

We are first interested in the modeling of the influence between two given motifs, a motif being defined as a sequence of letters in the alphabet \verb"{a,c,g,t}". This alphabet represents the four nucleotides bases of DNA: adenine, cytosine, guanine and thymine. Our aim is to model the dependence between motifs in order to identify favored or avoided distances between them, suggesting possible interactions at a molecular level. Because genomes are long (some 1 million bases) and motifs of interest are short (3 up to 20 bases), motif occurrences can be viewed as points along genomes. For convenience, we work in a continuous framework and then, the occurrences of a motif along a genome are modeled by a point process lying in the interval $[0;T]$, where $T$ is the normalized length of the studied genome and will drive the asymptotic. We add that our model focuses on only one direction of interactions, that is to say we investigate the way a first given motif influences a second one. To study the influence of the second motif on the first one, we just invert their roles in the model.

We observe the occurrences of both given motifs (we presuppose interactions between them) and we assume that their distributions are as follows.
The locations of the first motif are modeled by a $n$-sample of uniform random variables on $[0;T]$, denoted $U_{1},\ldots,U_{n}$ and named parents. As the parameter $T$, the number $n$ of parents will also drive the asymptotic.
Then, each $U_{i}$ gives birth independently to a Poisson process $N^i$ with intensity the function $t \longmapsto h(t-U_i)$ with respect to the Lebesgue measure on $\R$ (for instance, see \cite{Kin}), which models the locations of the second motif. We consequently observe the aggregated process
\begin{equation}
N = \sum_{i=1}^{n} N^i \quad \mbox{with intensity the function} \quad t \longmapsto \sum_{i=1}^{n} h(t-U_i)  \label{def_aggrproc}
\end{equation}
and the points of the process $N$ are named children.
But in this model, for any child we do not observe which parent gives birth to him.
The unknown function $h$ is so-called reproduction function. Our goal is then to estimate $h$ with the observations of the $U_i$'s and realizations of $N$.

Such a modeling of locations of the first motif is linked to the work on the distribution of words in DNA sequences of Schbath and coauthors (for instance, see \cite{Sch}, \cite{ReS} and \cite{RRS}). Indeed, the first motif of interest is a rare word and is modeled by a homogeneous Poisson process $N^0$ on $[0;T]$. Thus, conditionally to the event "the number of points falling into $[0;T]$ is $n$", the points of the process $N^0$ (i.e.\,the parents) obey the same law as a $n$-sample of uniform random variables on $[0;T]$. Moreover, with very high probability, $n$ is proportional to $T$ and this constitutes the asymptotic considered in genomics, to which we will refer as the "DNA case".
With our model (considering a uniform law on the parents), we can also take into consideration the cases $n \ll T$ (parents are far away with respect to each other and one can almost identify which points are the children of a given parent) and $n \gg T$ (parents are too close to each other, which leads to hard statistical problems).

If $n=1$, the purpose is to estimate the intensity of only one Poisson process. Many adaptive methods have been proposed to deal with Poisson intensity estimation.
For instance, Rudemo \cite{Rud} studied data-driven histogram and kernel estimates based on the cross-validation method.
Donoho \cite{Don} fitted the universal thresholding procedure proposed by Donoho and Johnstone \cite{DJ} by using the Anscombe's transform.
Kolaczyk \cite{Kol} refined this idea by investigating the tails of the distribution of the noisy wavelet coefficients of the intensity.
By using model selection, other optimal estimators have been proposed by Reynaud-Bouret \cite{Rey2} or Willett and Nowak \cite{WN}.
Reynaud-Bouret and Rivoirard \cite{RBR} proposed a data-driven thresholding procedure that is near optimal under oracle and minimax points of view, with as few support assumptions as possible (the support of the intensity $h$ may be unknown or not finite), unlike previous methods that need to assume that the intensity has a known bounded support.

We notice that the reproduction function $h$ can be also viewed as the intensity of a Cox process (for instance, see \cite{DVJ1}) where the covariates are the parents $U_1,\ldots,U_n$. Comte \emph{et al.}\,\cite{CGG} proposed an original estimator of the conditional intensity of a Cox process (more generally, a marker-dependent counting process). Using model selection methods, they prove that their estimator satisfies an oracle inequality and has minimax properties.
Note that we consider here point processes on the real line. Some aspects of similar spatial processes are studied in a parametric way \cite{Ras}, for instance.

Some work has been done to study the statistical dependence between motif occurrences.
For instance, in Gusto and Schbath's article \cite{GS}, the framework consists in modeling the occurrences of two motifs by a Hawkes process (see \cite{Haw}): our framework can be viewed when the support of $h$ is in $\R_+$ as a very particular case of theirs. Their method, called FADO, uses maximum likelihood estimates of the coefficients of $h$ on a Spline basis coupled with an AIC criterion. However, even if the FADO procedure is quite effective and can manage interactions between two types of events, spontaneous apparition (a child can be an orphan) and self-excitation (a child can give birth to another child), there are several drawbacks. In fact, this procedure is a parametric estimation method coupled with a classical AIC criterion which behaves poorly for complex families of models. Moreover, FADO involves sparsity issues. Indeed, our feeling is that if interaction exists, say around the distance $d$ bases, the function $h$ to estimate should take large values around $d$ and if there is no biological reason for any other interaction, then $h$ should be null anywhere else. However, if the FADO estimate takes small values in this last situation, it does not vanish, which can result in misleading biological interpretations (see \cite{RBS}). Finally, in this model, the occurrences of the first motif only depend on the past occurrences of first and second motifs.
Reynaud-Bouret and Schbath \cite{RBS} have proposed an alternative based on model selection principle for Hawkes processes that solves the sparsity problem. Their estimate satisfies an oracle inequality and has adaptive minimax properties with respect to certain classes of functions. But it manages only one motif whereas it is necessary to treat interaction with another type of events and the method has a high computational cost even for a small number of models.
Note that Hawkes processes have a long story of parametric inference (see \cite{Oza}, \cite{OA} and \cite{DVJ1}).
In particular, for genomic data, Carstensen \emph{et al.}\,\cite{CSWH} recently deal with multivariate Hawkes process models in order to model the joint occurrences of multiple transcriptional regulatory elements (TREs) along the genome that are capable of providing new insights into dependencies among elements involved in transcriptional regulation.

In this paper, the proposed model is simple. Each child comes from one parent (no orphan and no child who is a parent), that is to say we do not take into account the phenomenons of spontaneous apparition and self-excitation, contrary to Hawkes process models. But it brings novelties.
To estimate the reproduction function $h$, we propose a nonparametric method, using a wavelet thresholding rule that will compensate sparsity issues of the FADO method.
Furthermore, our model treats interaction between two types of events, with a possible influence of the past occurrences but also future occurrences.
Then, there is the presence of a double asymptotic: the normalized length of the studied genome $T$ and the number $n$ of parents, which is not usual.
In the biological context, it is not acceptable assuming to know each child's parent. Our model, via the reproduction function $h$, allows to quantify the favored locations of children in relation to their parent, even if one cannot attribute a child to a parent before the statistical inference.
First we provide in this paper theoretical results and we derive oracle inequalities and minimax rates showing that our method achieves good theoretical performances. The proofs of these results are essentially based on concentration inequalities and on exponential and moment inequalities for $U$-statistics (see \cite{dPG}, \cite{GLZ} and \cite{HRB}).
Secondly some simulations are carried out to validate our procedure and an application on real data (\emph{Escherichia coli} genome) is proposed. The procedure provides satisfying reconstructions, overcomes the problems raised by the FADO method and agrees with the knowledge of the considered biological mechanism. For these numerical aspects, we have used a low computational complexity cascade algorithm.

In Section 2, we define the notations and we describe the method. Then Section 2 discusses the properties of our procedure for the oracle and minimax approaches.
Section 3 is devoted to the implementation of our method and provides simulations. The cascade algorithm is presented in Section 3.1.
Section 4 presents the application on the complete \emph{Escherichia coli} genome.
A more technical result that is at the origin of the one stated in Section 2.3 and proofs can be found in Section 6 (Appendix).

\section{General results}

\subsection{Notations}

To estimate the reproduction function, we assume that $h$ belongs to $\L_1(\R)$ and  $\L_\infty(\R)$. Consequently, we can consider the decomposition of $h$ on a particular biorthogonal wavelet basis, built by Cohen \emph{et al.}\,\cite{CDF}, that we can describe as follows. We set $\phi=\1_{[0,1]}$ the analysis father wavelet. For any $r>0$, there exist three functions $\psi$, $\tilde\phi$ and $\tilde\psi$ with the following properties:
\begin{itemize}
  \item $\tilde\phi$ and $\tilde\psi$ are compactly supported,
  \item $\tilde\phi$ and $\tilde\psi$ belong to $C^{r+1}$, where $C^{r+1}$ denotes the Hölder space of order $r+1$,
  \item $\psi$ is compactly supported and is a piecewise constant function,
  \item $\psi$ is orthogonal to polynomials of degree no larger than $r$,
  \item $\left\{(\phi_k,\psi_{j,k})_{j\geq 0,k\in\Z},(\tilde\phi_k,\tilde\psi_{j,k})_{j\geq 0,k\in\Z}\right\}$ is a biorthogonal family:
      for any $j,j'\geq 0$, for any $k,k'\in\Z$,
      \[\int_\R \phi_k(x) \tilde\psi_{j',k'}(x) \,dx = \int_\R \psi_{j,k}(x) \tilde\phi_{k'}(x) \,dx = 0,\]
      \[\int_\R \phi_k(x) \tilde\phi_{k'}(x) \,dx = \1_{\{k=k'\}}, \quad \int_\R \psi_{j,k}(x) \tilde\psi_{j',k'}(x) \,dx = \1_{\{j=j',k=k'\}},\]
      where for any $x\in\R$,
      \[\phi_{k}(x) = \phi(x-k), \quad \psi_{j,k}(x) = 2^{j/2} \psi(2^jx-k)\]
      and
      \[\tilde\phi_{k}(x) = \tilde\phi(x-k), \quad \tilde\psi_{j,k}(x) = 2^{j/2} \tilde\psi(2^jx-k).\]
\end{itemize}
On the one hand, decomposition wavelets $\phi_k$ and $\psi_{j,k}$ are piecewise constant functions and, on the other hand, reconstruction wavelets $\tilde\phi_k$ and $\tilde\psi_{j,k}$ are smooth functions.
This implies the following wavelet decomposition of $h \in \L_2(\R)$:
\begin{equation}
h = \sum_{k\in\Z} \alpha_k \tilde\phi_k + \sum_{j\geq 0} \sum_{k\in\Z} \beta_{j,k} \tilde\psi_{j,k},  \label{dec0}
\end{equation}
where for any  $j\geq 0$ and any $k\in\Z$,
\[\alpha_k = \int_\R  h(x) \phi_k(x) \,dx, \quad \beta_{j,k} = \int_\R h(x) \psi_{j,k}(x) \,dx.\]
The Haar basis, used in practice, can be viewed as a particular biorthogonal wavelet basis, by setting $\tilde\phi=\phi$ and
$\tilde\psi=\psi=\1_{]\frac12;1]}-\1_{[0;\frac12]}$, with $r=0$ (even if the second property is not satisfied with such a choice). The Haar basis is an orthonormal basis, which is not true for general biorthogonal wavelet bases.
This kind of decomposition has already been used in thresholding methods by Juditsky and Lambert-Lacroix \cite{JLL}, Reynaud-Bouret and Rivoirard \cite{RBR}, and Reynaud-Bouret \emph{et al.}\,\cite{RBRTM}.

To shorten mathematical expressions, we set
\[\Lambda = \{\lambda = (j,k) : j \geq -1, k \in \Z\},\]
and for any $\lambda\in\Lambda$,
\[\f = \left\{\begin{array}{cl} \phi_k & \mbox{if $\lambda = (-1,k)$} \\ \psi_{j,k} & \mbox{if $\lambda = (j,k)$ with $j \geq 0$} \end{array}\right., \quad \tilde\varphi_{\lambda} = \left\{\begin{array}{cl} \tilde\phi_k & \mbox{if $\lambda = (-1,k)$} \\ \tilde\psi_{j,k} & \mbox{if $\lambda = (j,k)$ with $j \geq 0$} \end{array}\right.\]
and similarly
\[\beta_{\lambda} = \left\{\begin{array}{cl} \alpha_k & \mbox{if $\lambda = (-1,k)$} \\ \beta_{j,k} & \mbox{if $\lambda = (j,k)$ with $j \geq 0$} \end{array}\right..\]
Then (\ref{dec0}) can be rewritten as
\begin{equation}
h = \sum_{\lambda\in\Lambda} \beta_{\lambda} \tilde\varphi_{\lambda} \quad \mbox{with} \quad \beta_{\lambda} = \int_\R h(x)\f(x) \,dx  \label{dec}
\end{equation}
and now, we have to estimate these wavelet coefficients.

For all $\lambda$ in $\Lambda$, we define $\hat{\beta}_{\lambda}$ an estimator of $\beta_{\lambda}$ as
\begin{equation}
\hat{\beta}_{\lambda} = \frac{G(\f)}{n}, \quad \mbox{with} \quad G(\f) = \int_\R \sum_{i=1}^{n} \left[\f(t-U_i) - \frac{n-1}{n} \E_{\pi}(\f(t-U))\right] \,dN_{t},  \label{def_est}
\end{equation}
where $\pi$ is the uniform distribution on $[0;T]$ and $\E_{\pi}(\f(t-U))$ denotes the expectation of $\f(t-U)$ where $U \sim \pi$ (an independent copy of $U_1,\ldots,U_n$).
If $n=1$, we obtain the natural estimators of the $\beta_{\lambda}$'s in the case of only one Poisson process on the real line (see \cite{RBR}).

\begin{lem}  \label{lemEST}
For all $\lambda = (j,k)$ in $\Lambda$,
\[\E(G(\f)) = n \int_\R \f(x)h(x) \,dx,\]
i.e.\,$\hat{\beta}_{\lambda}$ is an unbiased estimator for $\beta_{\lambda}$.
Furthermore, its variance is upper bounded as follows:
\[\var(\hat{\beta}_{\lambda}) \leq C \left\{\frac{1}{n} + \frac{1}{T} + \frac{2^{-j}n}{T^2}\right\}\]
and
\[\sup_{\lambda\in\Lambda} \var(\hat{\beta}_{\lambda}) \leq C' \left\{\frac{1}{n} + \frac{n}{T^2}\right\},\]
where $C$ and $C'$ depend on $\|h\|_1$, $\|h\|_{\infty}$, $\|\psi\|_1$ and $\|\psi\|_2$.
\end{lem}

The behavior of the variance of the $\hat{\beta}_{\lambda}$'s is not usual, because two parameters $n$ and $T$ are involved.
Nevertheless, when $n$ is proportional to $T$ ("DNA case" as explained in Introduction), the variance is bounded by $1/T$ up to a constant, as for the Hawkes process (see \cite{RBS}).
When $n \ll T$, the variance is bounded by $1/n$ up to a constant, which means that the apparition's distance between two parents is large enough to make their interactions insignificant for the statistical analysis. So in this case, our framework can be viewed as the observation of a $n$-sample of a Poisson process with common intensity $h$ (see \cite{RBR}).
Finally, when $n \gg T$, the variance deteriorates and is only bounded by $n/T^2$ up to a constant, and in this case, the small apparition's distance between two parents leads to rough statistical issues hard to overcome.

\subsection{Description of our method}

We start assuming that $h$ is compactly supported in $[-A;A]$, with $A$ a positive real number. This quantity $A$ can denote the maximal memory along DNA sequences (this is chosen by the biologists (see \cite{GS}), depending on the underlying biological process they have in mind).
Furthermore, the properties of the biorthogonal wavelet bases introduced previously allow us to assume that we know a positive real number $M$ such that the support of $\psi$ is contained in $[-M;M]$.

First, we introduce the following deterministic subset $\Gamma$ of $\Lambda$
\[\Gamma = \big\{\lambda = (j,k) \in \Lambda : -1 \leq j \leq j_0, k \in \mathcal{K}_j\big\},\]
where $j_0$ a positive integer that will be fixed later and at each resolution level $j$, we denote $\mathcal{K}_j$ the set of integers such that the intersection of the supports of $\f$ and $h$ is not empty, with $\lambda=(j,k)$.
Straightforward computations lead to a cardinal of $\Gamma$ of order $2^{j_0}$.

Then, given some parameter $\gamma > 0$, we define for any $\lambda\in\Gamma$, the threshold
\begin{equation}
\eta_\lambda(\gamma,\Delta) = \sqrt{2\gamma j_0 \widetilde{V}\left(\frac{\f}{n}\right)} + \frac{\gamma j_0}{3} B\left(\frac{\f}{n}\right) + \Delta \frac{N_\R}{n}  \label{def_threshold}
\end{equation}
where $\Delta$ is a positive quantity and $N_{\R}$ is the number of points of the aggregated process $N$ lying in $\R$.
For theoretical results, $\Delta$ will be taken of order $\frac{j_0^2 2^{j_0/2}}{n} + \frac{j_0}{\sqrt{T}} + \frac{\sqrt{j_0n}}{T}$ times a constant depending on $\gamma$, $\|\psi\|_1$, $\|\psi\|_2$ and $\|\psi\|_{\infty}$.
In (\ref{def_threshold}), we set
\begin{equation}
B\left(\frac{\f}{n}\right) = \frac{1}{n} B(\f) = \frac{1}{n} \left\| \sum_{i=1}^{n} \left[\f(\cdot-U_i) - \frac{n-1}{n} \E_{\pi}(\f(\cdot-U))\right] \right\|_{\infty}  \label{def_B}
\end{equation}
and
\begin{equation}
\widetilde{V}\left(\frac{\f}{n}\right) = \frac{1}{n^2} \widetilde{V}(\f) = \frac{1}{n^2} \left(\hat{V}(\f) + \sqrt{2\gamma j_0\hat{V}(\f)B^2(\f)} + 3\gamma j_0B^2(\f)\right)  \label{def_tildeV}
\end{equation}
where
\begin{equation}
\hat{V}(\f) = \int_\R \left(\sum_{i=1}^{n} \left[\f(t-U_i) - \frac{n-1}{n} \E_{\pi}(\f(t-U))\right]\right)^2 \,dN_{t}.  \label{def_hatV}
\end{equation}
Since they only depend on the observations, the numerical values of $B(\f)$, $\hat{V}(\f)$ and so $\widetilde{V}(\f)$ defined respectively by (\ref{def_B}), (\ref{def_hatV}) and (\ref{def_tildeV}) can be exactly computed.

We denote $\tilde{\beta}$ the estimator of $\beta = (\beta_\lambda)_{\lambda\in\Lambda}$ associated with the previous thresholding rule:
\begin{equation}
\tilde{\beta} = \left( \hat{\beta}_{\lambda} \1_{|\hat{\beta}_{\lambda}|\geq\eta_{\lambda}(\gamma,\Delta)} \1_{\lambda\in\Gamma} \right)_{\lambda\in\Lambda}  \label{def_tildebeta}
\end{equation}
and finally, we set
\begin{equation}
\tilde{h} = \sum_{\lambda\in\Lambda} \tilde{\beta}_{\lambda} \tilde\varphi_{\lambda}  \label{def_tildeh}
\end{equation}
an estimator of $h$ that only depends on the choice of $(\gamma,\Delta)$ and $j_0$ fixed later.

Thresholding procedures have been introduced by Donoho and Johnstone \cite{DJ}. They derive from the sufficiency to keep a small amount of the coefficients to have a good estimation of the function $h$. The threshold $\eta_\lambda(\gamma,\Delta)$ seems to be defined in a rather complicated manner but the first term: $\sqrt{2\gamma j_0 \widetilde{V}\left(\frac{\f}{n}\right)}$ looks like the universal threshold proposed by \cite{DJ} in the Gaussian regression framework, by choosing $\gamma$ close to 1 and $j_0$ of order $\log{n}$. The universal threshold of \cite{DJ} is defined by $\eta_\lambda = \sqrt{2\sigma^2\log{n}}$, where $\sigma^2$ (assumed to be known) is the variance of each noisy wavelet coefficient. In our setting, $\var(\hat{\beta}_{\lambda})$ depends on $h$, so it is (over)estimated by $\widetilde{V}\left(\frac{\f}{n}\right)$. The other terms of the threshold (\ref{def_threshold}) are unavoidable remaining terms which allow to obtain sharp concentration inequalities.

\subsection{Main result and discussions}

Our main result is an oracle one. Given a collection of procedures (for example, penalization, projection or thresholding), the oracle represents the ideal "estimator" among the collection. In our setting the oracle gives, for our thresholding rule, the coefficients that have to be kept. In our framework (see \cite{DJ} and \cite{RBR}), the "oracle estimator" is
\[\bar{h} = \sum_{\lambda\in\Gamma} \bar{\beta}_\lambda \tilde\varphi_\lambda, \quad \mbox{with} \quad \bar{\beta}_\lambda = \hat{\beta}_\lambda \1_{\var(\hat{\beta}_\lambda) < \beta_\lambda^2}.\]
This "estimator" is not a true estimator, of course, since it depends on $h$. The approach of optimal adaptation is to derive true estimators which achieve the same performance as the "oracle estimator".
Our goal is now to compare the risk of $\tilde h$ defined in Section 2.2 to the oracle risk:
\[\E\left(\|\bar{h}-h\|_2^2\right) = \sum_{\lambda\in\Gamma} \E\big[(\hat{\beta}_\lambda \1_{\var(\hat{\beta}_\lambda) < \beta_\lambda^2} - \beta_\lambda)^2\big] + \sum_{\lambda\not\in\Gamma} \beta_\lambda^2 = \sum_{\lambda\in\Gamma} \min(\var(\hat{\beta}_\lambda),\beta_\lambda^2) + \sum_{\lambda\not\in\Gamma} \beta_\lambda^2.\]

\begin{thm}  \label{thm1}
We assume that $n \geq 2$, $j_0\in\N^*$ such that $2^{j_0} \leq n < 2^{j_0+1}$, $\gamma > 2\log2$ and $\Delta$ is defined in the Appendix by (\ref{def_Delta}) and (\ref{def_d}).
Then the estimator $\tilde{h}$ defined in Section 2.2 satisfies
\[\E\left(\|\tilde{h}-h\|_2^2\right) \leq C_1 \inf_{m \subset \Gamma} \Bigg\{\sum_{\lambda \not\in m} \beta_{\lambda}^{2} + \Bigg[(\log{n})^4 \times \frac{1}{n} + (\log{n})^2 \times \frac{n}{T^2}\Bigg] |m|\Bigg\} + C_2 \Bigg[\frac{1}{n} + \frac{n}{T^{2}}\Bigg],\]
where $|m|$ is the cardinal of the set $m$, $C_1$ is a positive constant depending on $\gamma$, $\|h\|_1$, $\|h\|_{\infty}$, $\|\psi\|_1$, $\|\psi\|_2$ and $\|\psi\|_{\infty}$ and $C_2$ is a positive constant depending on the compact support of $h$, $\|h\|_1$, $\|h\|_{\infty}$, the compact support of $\psi$, $\|\psi\|_1$, $\|\psi\|_2$ and $\|\psi\|_{\infty}$.
\end{thm}

As the expression between brackets is of the same order as the upper bound of $\var(\hat{\beta}_{\lambda})$ established in \ref{lemEST} (up to a logarithmic term), the oracle type inequality of \ref{thm1} proves that the estimator $\tilde{h}$ achieves satisfying theoretical properties.

In particular, if we apply \ref{thm1} with $n$ proportional to $T$ ("DNA case"), then the estimator $\tilde{h}$ defined in Section 2.2 satisfies
\[\E\left(\|\tilde{h}-h\|_2^2\right) \leq C_1 \inf_{m \subset \Gamma} \Bigg\{\sum_{\lambda \not\in m} \beta_{\lambda}^{2} + \frac{(\log{T})^4}{T} |m|\Bigg\} + \frac{C_2}{T}.\]
This oracle type inequality is similar to the one obtained by Theorem 1 of \cite{RBS} where the Hawkes process is considered. Since $n$ is proportional to $T$, this inequality is typical of classical oracle inequalities obtained in model selection (for example, see Theorem 2.1 of \cite{RBR} where only one Poisson process on the real line is considered or more generally, see \cite{Mas} for density estimation).

Then, we establish a minimax result on Besov balls still with $n$ is proportional to $T$.
For any $R>0$ and $s\in\R$ such that $0<s<r+1$ (where $r>0$ denotes the wavelet smoothness parameter introduced in the description of the biorthogonal wavelet bases at the beginning of the current section), we consider the following Besov ball of radius $R$:
\[\mathcal{B}^{s}_{2,\infty}(R) = \left\{f \in \L_2(\R) : f = \sum_{\lambda\in\Lambda} \beta_{\lambda} \tilde\varphi_{\lambda}, \forall j \geq -1, \sum_{k \in \mathcal{K}_j} \beta_{(j,k)}^2 \leq R^2 2^{-2js}\right\}.\]
Now, let us state the upper bound of the risk of $\tilde{h}$ when $h$ belongs to $\mathcal{B}^{s}_{2,\infty}(R)$.

\begin{cor}
Let $R>0$ and $s\in\R$ such that $0<s<r+1$. Assume that $h\in\mathcal{B}^{s}_{2,\infty}(R)$ and $n$ is proportional to $T$.
Then the estimator $\tilde{h}$ defined in Section 2.2 satisfies
\[\E\left(\|\tilde{h}-h\|_2^2\right) \leq C \left(\frac{(\log{T})^4}{T}\right)^{\frac{2s}{2s+1}},\]
where $C$ is a positive constant depending on $\gamma$, the compact support of $h$, $\|h\|_1$, $\|h\|_{\infty}$, the compact support of $\psi$, $\|\psi\|_1$, $\|\psi\|_2$, $\|\psi\|_{\infty}$ and $R$.
\end{cor}

The rate of the risk of $\tilde{h}$ corresponds to the minimax rate, up to the logarithmic term, for estimation of a compactly supported intensity of a Poisson process (see \cite{Rey2}) or for a compactly supported density when we have $n$ i.i.d.\,observations (see \cite{DJKP}). One more time this illustrates the optimality of the procedure $\tilde{h}$ but in the minimax setting.

\section{Implementation procedure}

From now on we consider the context of DNA, i.e.\,$n$ is proportional to $T$. As mentioned in Introduction, we can assume that the parents are the points of a homogeneous Poisson process $N^0$ on $[0;T]$ with constant intensity $\mu$ which allows to write $n \simeq \mu T$.

In this section, we specify a procedure for the computation of the family of random thresholds $(\eta_\lambda(\gamma,\Delta))_{\lambda\in\Gamma}$ to reconstruct the reproduction function $h$. We also provide some simulations in order to calibrate parameters from a numerical point of view and to show the robustness of our procedure.

\subsection{Algorithm}

We only focus on the Haar basis where
\[\phi=\tilde\phi=\1_{[0;1]} \quad \mbox{and} \quad \psi=\tilde\psi=\1_{]\frac12;1]}-\1_{[0;\frac12]},\]
because the expression of the functions associated to this basis, that are piecewise constant functions, allows to implement simple and fast algorithms. Furthermore, considering this kind of functions is suitable for our genomic setting. In fact, according to biological studies, the reproduction function $h$ is expected to be very irregular, with large null ranges and sudden changes at specific distances.
We recall that $h$ is assumed to be compactly supported in $[-A;A]$, with $A$ a positive integer in practice.

We consider the thresholding rule $\tilde{h}$ defined in Section 2.2 with
\[\Gamma = \big\{\lambda = (j,k) \in \Lambda : -1 \leq j \leq j_0, k \in \mathcal{K}_j\big\},\]
and
\[\eta_\lambda(\gamma,\delta) = \sqrt{2\gamma j_0 \hat{V}\left(\frac{\f}{n}\right)} + \frac{\gamma j_0}{3} B\left(\frac{\f}{n}\right) + \frac{\delta}{\sqrt{T}} \frac{N_\R}{n}.\]
Observe that $\eta_\lambda(\gamma,\delta)$ slightly differs from the threshold defined in (\ref{def_threshold}) since the parameter $\Delta$ is replaced with $\frac{\delta}{\sqrt{T}}$ (thanks to the definition (\ref{def_Delta}) of $\Delta$) and $\widetilde{V}(\f)$ is now replaced with $\hat{V}(\f)$ (there is no major difference in our simulations).
The ideal choice (from a theoretical point of view) of the maximal resolution level $j_0$ is given by \ref{thm1}, that is to say $j_0$ is the positive integer such that $2^{j_0} \leq n < 2^{j_0+1}$.
But we will fix $j_0=5$ in the sequel (in particular, to limit the computation time).
The choice of the parameters $\gamma$ and $\delta$ is discussed in the next subsection.

A key point of the algorithm is the computation of the quantity
\[S(\f)(t) = \sum_{i=1}^{n} \left[\f(t-U_i) - \frac{n-1}{n} \E_{\pi}(\f(t-U))\right], \quad \mbox{for all $t \in \R$}\]
that appears in $\hat{\beta}_{\lambda}$, $B(\f)$ and $\hat{V}(\f)$.
We decompose it into two parts: a random "piecewise constant" part $\displaystyle S_r(\f) = \sum_{i=1}^{n} \f(\cdot-U_i)$ and a deterministic (piecewise affine) part $(n-1) \E_{\pi}(\f(t-U))$.
Note that the deterministic part can easily be implemented with a low computational cost.
This is not the case of the random "piecewise constant" part for which we have constructed a cascade algorithm, inspired by the pioneering work of Mallat \cite{Mal}.
To explain this algorithm in few words, we use the following notations: for any $j \geq 0$, for any $k\in\Z$, for any $x\in\R$,
\[\phi_{j,k}(x) = 2^{j/2} \phi(2^jx-k), \quad \psi_{j,k}(x) = 2^{j/2} \psi(2^jx-k),\]
where $\phi_{j,k}$ are father wavelets and $\psi_{j,k}$ mother wavelets. We have the following relationships between wavelets at level $j$ and wavelets at level $(j+1)$:
\begin{equation}
\psi_{j,k} = \frac{\sqrt{2}}{2}\big(\phi_{j+1,2k+1} - \phi_{j+1,2k}\big) \quad \mbox{and} \quad \phi_{j,k} = \frac{\sqrt{2}}{2}\big(\phi_{j+1,2k} + \phi_{j+1,2k+1}\big).  \label{rel_wavelet}
\end{equation}
We notice that only mother wavelets and the father wavelet of level $j=0$ (that corresponds to $\f$ with $\lambda = (-1,k)$) are used to reconstruct the signal.
The cascade algorithm is implemented as follows.
\begin{enumerate}
  \item Compute $S_r(\phi_{j_0,0})$. Since $S_r(\phi_{j_0,0})$ is a piecewise constant function, this computation gives a partition and the values of $S_r(\phi_{j_0,0})$ on the intervals of the partition.
  \item Shift by $+2^{-j_0}k$ the intervals of the previous partition by keeping the same values on the partition to obtain $S_r(\phi_{j_0,k})$ for any integer $k$ in $[-2^{j_0}A;2^{j_0}A-1]$.
  \item For any resolution level $j$ going from $(j_0-1)$ to 0, in a decreasing way, compute $S_r(\psi_{j,k})$ and $S_r(\phi_{j,k})$ with expressions (\ref{rel_wavelet}). The quantities $S_r(\psi_{j,k})$ allow the reconstruction of the signal and the quantities $S_r(\phi_{j,k})$ are transitional and will be used for the computations of the lower resolution level $(j-1)$.
  \item Also keep $S_r(\phi_{0,k})$ because it is used for the reconstruction of the signal.
\end{enumerate}

Now, let us define our thresholding estimate of $h$ for a practical purpose.
\begin{description}
  \item[Step 0] Let $j_0=5$ and choose also positive constants $\gamma$ and $\delta$.
  \item[Step 1] Set $\Gamma = \big\{\lambda = (j,k) \in \Lambda : -1 \leq j \leq j_0, k \in \mathcal{K}_j\big\}$ and compute for any $\lambda$ in $\Gamma$, $S(\f)(X)$ for all points $X$ of the process $N$. In the same way, also compute the coefficients $\hat{\beta}_{\lambda}$, $B(\f)$ and $\hat{V}(\f)$.
  \item[Step 2] Threshold the coefficients by setting $\tilde{\beta}_{\lambda} = \hat{\beta}_{\lambda} \1_{|\hat{\beta}_{\lambda}|\geq\eta_{\lambda}(\gamma,\delta)}$ according to the following threshold choice:
      \[\eta_\lambda(\gamma,\delta) = \sqrt{2\gamma j_0 \hat{V}\left(\frac{\f}{n}\right)} + \frac{\gamma j_0}{3} B\left(\frac{\f}{n}\right) + \frac{\delta}{\sqrt{T}} \frac{N_\R}{n}.\]
  \item[Step 3] Reconstruct the function $h$ by using the $\tilde\beta_\lambda$'s and denote
      \[\tilde{h} = \sum_{\lambda\in\Lambda} \tilde{\beta}_{\lambda} \tilde\varphi_{\lambda}.\]
\end{description}

\subsection{Experiments on simulated data}

The programs have been coded in \texttt{Scilab 5.2} and are available upon request.

\subsubsection{Choice of parameters}

Now, we deal with the choice of the parameters $\gamma$ and $\delta$ in our procedure from a practical point of view. The question is: how to choose the optimal parameters? We work with two testing functions denoted 'Signal1' and 'Signal2' whose definitions are given in the following table:
\begin{center}
\begin{tabular}{|c|c|}
\hline & \\
'Signal1' & 'Signal2' \\
& \\
$\nu \times \1_{[0;1]}$ & $\nu\times\frac{8}{3}\left(\1_{[0.5;0.625]}+\1_{[1;1.25]}\right)$ \\
& \\ \hline
\end{tabular}
\end{center}
with $\nu$, the children's intensity, set to 4.
We fix willfully $A=10$. Such a choice of $A$ (remember that $[-A;A]$ is the support of $h$) assumes that we do not know the support of functions.
We recall that $j_0=5$.

Given $T$, $\mu$ the parents' intensity and a testing function, we denote $R(\gamma,\delta)$ the quadratic risk of our procedure $\tilde{h}$ (depending on $(\gamma,\delta)$) defined in Section 3.1. Of course, we aim at finding values of $(\gamma,\delta)$ such that this quadratic risk is minimal. The average over 100 simulations of $R(\gamma,\delta)$ is computed providing an estimation of $\E(R(\gamma,\delta))$. This average risk, denoted $\bar{R}(\gamma,\delta)$ and viewed as a function of the parameters $(\gamma,\delta)$, is plotted for $(T,\mu) \in \{(10000,0.1),(2000,0.1),(2000,0.5)\}$ and for the two signals considered previously: 'Signal1' and 'Signal2'.

\begin{figure}[ht]
  \begin{center}
  \includegraphics[width=13cm,height=6cm]{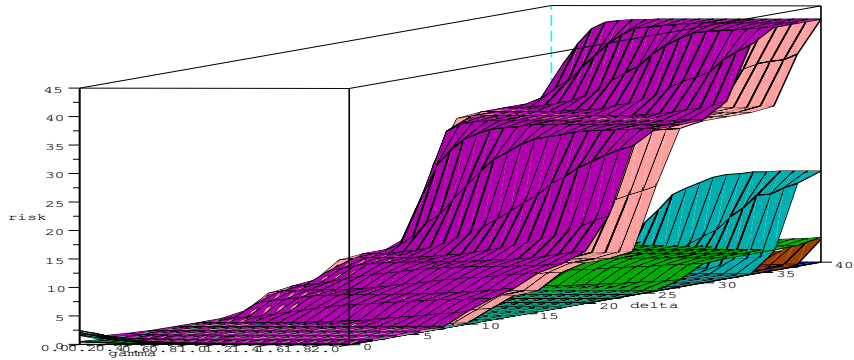}
  \caption{{\small\emph{The function $(\gamma,\delta) \mapsto \bar{R}(\gamma,\delta)$ for 'Signal1' and 'Signal2' for different values of $T$ and $\mu$: 'Signal1' in \textcolor[rgb]{0,0,0.69}{$\blacksquare$}  and 'Signal2' in \textcolor[rgb]{0,0.69,0}{$\blacksquare$} with $(T,\mu)=(10000,0.1)$; 'Signal1' in \textcolor[rgb]{0,0.69,0.69}{$\blacksquare$}  and 'Signal2' in \textcolor[rgb]{0.69,0,0.69}{$\blacksquare$} with $(T,\mu)=(2000,0.1)$; 'Signal1' in \textcolor[rgb]{0.63,0.25,0}{$\blacksquare$}  and 'Signal2' in \textcolor[rgb]{1,0.63,0.63}{$\blacksquare$} with $(T,\mu)=(2000,0.5)$}}.}
  \label{fig1}
  \end{center}
\end{figure}

\ref{fig1} displays $\bar{R}$ for 'Signal1' and 'Signal2' decomposed on the Haar basis.
This figure allows to draw the following conclusion: for any $(T,\mu) \in \{(10000,0.1),(2000,0.1),(2000,0.5)\}$ and for 'Signal1' or 'Signal2',
\[\bar{R}(\gamma,\delta)\approx 0\]
for many values of $(\gamma,\delta)$. So, we observe a kind of "plateau phenomenon".

Reconstructions of the intensities of 'Signal1' and 'Signal2' are respectively given in \ref{fig2} and \ref{fig3} with the choice $(\gamma,\delta)=(0.18,2.4)$, a common value of several plateaus.
Note the good performance of our thresholding rule, in particular for $T=10000$ and $\mu=0.1$ (we have $\mu T=1000$ parents and $\mu\nu T=4000$ children in average), which corresponds to the real case treated in Section 4.
Thus, we propose to take systematically $(\gamma,\delta)=(0.18,2.4)$ in our procedure $\tilde{h}$ defined in Section 3.1.

\begin{figure}[ht]
  \begin{center}
  \includegraphics[width=13cm,height=7cm]{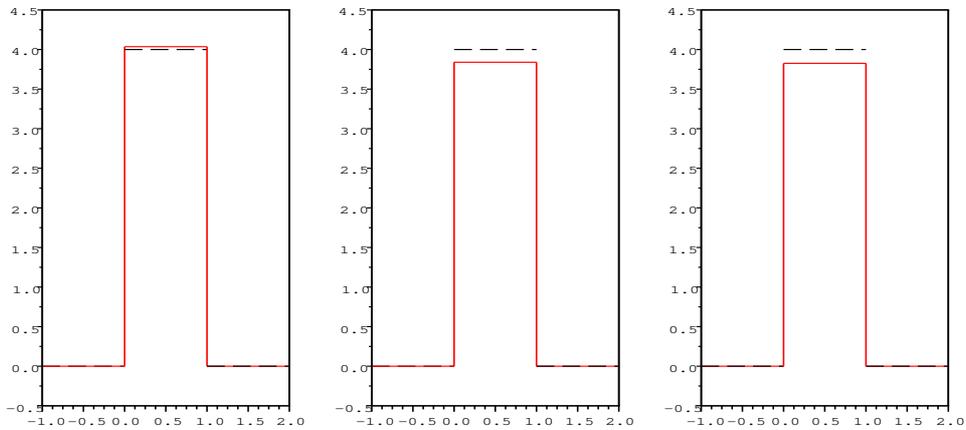}
  \caption{{\small\emph{Reconstructions of 'Signal1' (true: dotted line, estimate: solid line): left: $(T,\mu)=(10000,0.1)$; middle: $(T,\mu)=(2000,0.1)$; right: $(T,\mu)=(2000,0.5)$}}.}
  \label{fig2}
  \end{center}
\end{figure}

\begin{figure}[ht]
  \begin{center}
  \includegraphics[width=13cm,height=7cm]{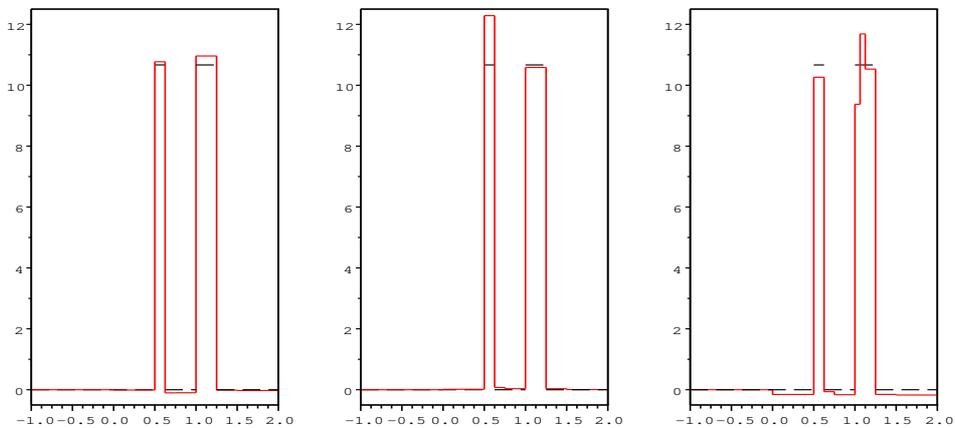}
  \caption{{\small\emph{Reconstructions of 'Signal2' (true: dotted line, estimate: solid line): left: $(T,\mu)=(10000,0.1)$; middle: $(T,\mu)=(2000,0.1)$; right: $(T,\mu)=(2000,0.5)$}}.}
  \label{fig3}
  \end{center}
\end{figure}

\subsubsection{About the support of $h$}

We are interested in the robustness of our procedure with respect to the support issue from a numerical point of view.
What happens if we are wrong about the support of the function that we want to estimate?
For instance, we consider the testing function denoted 'Signal3' whose definition is given in the following table:
\begin{center}
\begin{tabular}{|c|}
\hline \\
'Signal3' \\
\\
$\nu \times \frac{1}{4}\left(\1_{[-0.75;-0.5]}+\1_{[4.25;8]}\right)$ \\
\\ \hline
\end{tabular}
\end{center}
with $\nu$, the children's intensity, set to $4$.

\ref{fig4} displays reconstructions of 'Signal3' with different supports of $h$: $[-A;A]$, with $A\in\{1,5,10\}$.
This figure shows that when we take a not large enough support ($A=1$ or $5$), we do not make large errors of approximation on $[-A;A]$. So, the procedure seems to take into account what happens beyond the chosen support.
And for $A=10$, we have a good complete reconstruction of 'Signal3'.

\begin{figure}[ht]
  \begin{center}
  \includegraphics[scale=0.5]{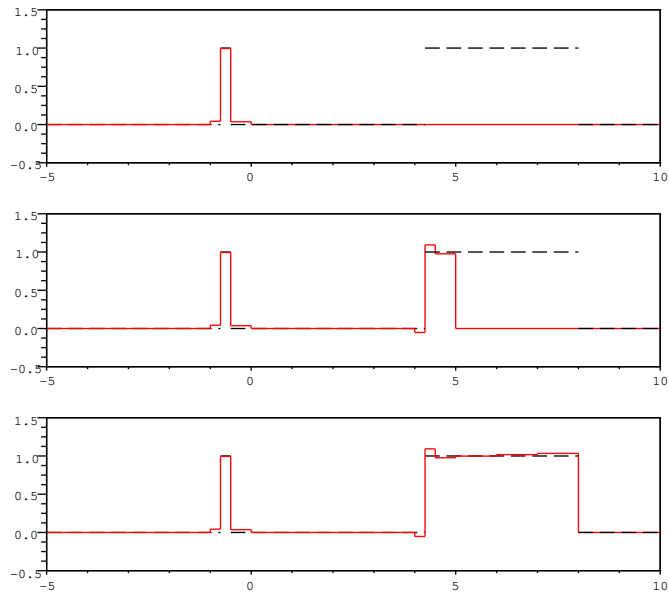}
  \caption{{\small\emph{Reconstructions of 'Signal3' (true: dotted line, estimate: solid line) with different supports: top: $A=1$; middle: $A=5$; bottom: $A=10$}}.}
  \label{fig4}
  \end{center}
\end{figure}

Finally, even if the support of the reproduction function is unknown, our method estimates correctly the signal on the chosen support, which explains the robustness of our procedure with respect to the support issue.

\subsubsection{The case of spontaneous apparition}

Here, we investigate the case of spontaneous apparition.
Even if our model does not take into account the spontaneous apparition (i.e.\,children can not be orphans), we are interested by the performance of our procedure if there is a presence of orphans.
On the one hand, let us give two processes: a process of intensity 'Signal1' with $\nu=3$, $T=10000$ and $\mu=0.1$, to which is added a homogeneous Poisson process on $[0;T+1]$ with intensity $\mu(4-\nu)=0.1$ (the orphans are viewed as a Poissonian noise). Thus, we have in average 1000 parents, 3000 children having a parent and 1000 children being orphans.
On the other hand, let us give two other processes: a process of intensity 'Signal1' with $\nu=1$ this time, $T=10000$ and $\mu=0.1$, to which is added a homogeneous Poisson process on $[0;T+1]$ with intensity $\mu(4-\nu)=0.3$. Thus, we have in average 1000 parents, 1000 children having a parent and 3000 children being orphans.

Reconstructions of 'Signal1' with $\nu=3$ and $\nu=1$ are given in \ref{fig5}.
When there is a small proportion of children being orphans, the reconstruction is still acceptable; the procedure can manage few orphans.
But, when there are too many orphans, our procedure makes approximation errors, which are due to the fact that our model consists in associating any child with a parent.

\begin{figure}[ht]
  \begin{center}
  \includegraphics[scale=0.5]{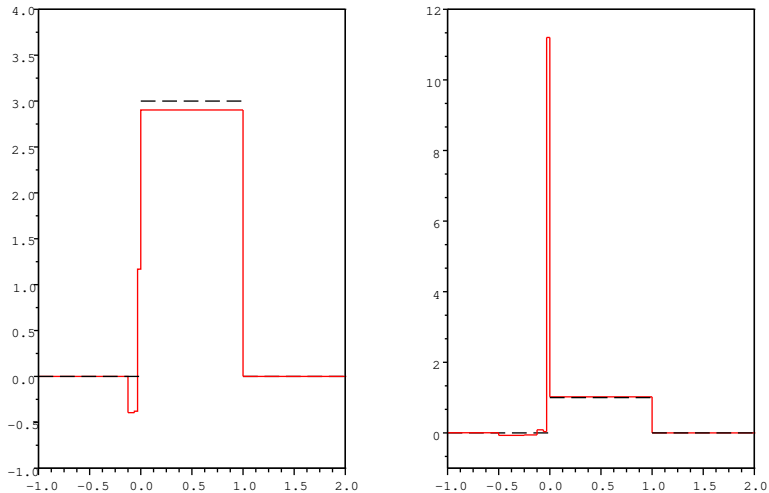}
  \caption{{\small\emph{Reconstructions of 'Signal1' (true: dotted line, estimate: solid line) with different values of $\nu$: left: $\nu=3$; right: $\nu=1$}}.}
  \label{fig5}
  \end{center}
\end{figure}

We mention that the case of spontaneous apparition is only numerical. For a more precise study of this phenomenon, we should extend our model by adding a positive constant to the intensity function $t \longmapsto \sum_{i=1}^{n} h(t-U_i)$, that would represent the orphans. This is outside the scope of this paper.

\section{Applications to genomic data}

As application, we are interested in the \emph{Escherichia coli} genome. \emph{E.\,coli} is an intestinal bacterium in mammals and very common in humans which is widely studied and used in genetics. More precisely, we are interested in the study of the dependence between promoter sites and genes along the complete genome of the bacterium. In particular, promoters are usually structured motifs located before the genes and not too far from them. Here, we have considered the major promoter of the bacterium \emph{E.\,coli} and more precisely the word \texttt{tataat}. Most of the genes of \emph{E.\,coli} should be preceded by this word at a very short distance apart. In order to validate our thresholding estimation procedure (proposed at Section 3), we hope to detect short favored distances between genes and previous occurrences of \texttt{tataat}.

For this, as in \cite{GS} we have analyzed the sequence composed of both strands of \emph{E.\,coli} genome (4639221 bases); each strand being separated by 10000 artificial bases to avoid artificial dependencies between occurrences on one strand and occurrences on the other strand; we took 10000 bases for the maximal memory. It then represents a sequence of length 9288442; there are 4290 genes (we took the positions of the first base of coding sequences) and 1036 occurrences of \texttt{tataat}. For convenience, we set $T=9289$ and so $A=10$ (we work on a scale of $1:1000$). We recall that we have fixed $j_0=5$ and taken $(\gamma,\delta)=(0.18,2.4)$.

First, we investigate the way the DNA motif \texttt{tataat} influences genes and so, in our model, the parents are the occurrences of \texttt{tataat} and children are the occurrences of genes.
To give general insight on $h$, \ref{fig6} gives the estimator $\tilde h$ defined in Section 3.1 without Step 2 (no thresholding), i.e.\,we have kept all the estimated coefficients.
We observe a peak around $0$ which corresponds to what we thought about the fact that most of the genes of \emph{E.\,coli} should be preceded by the word \texttt{tataat} at a very short distance apart. We also observe other peaks, for instance around 1200 bases. The biological significance of these peaks remains an open question.

\begin{figure}[ht]
  \begin{center}
  \includegraphics[width=13cm,height=6cm]{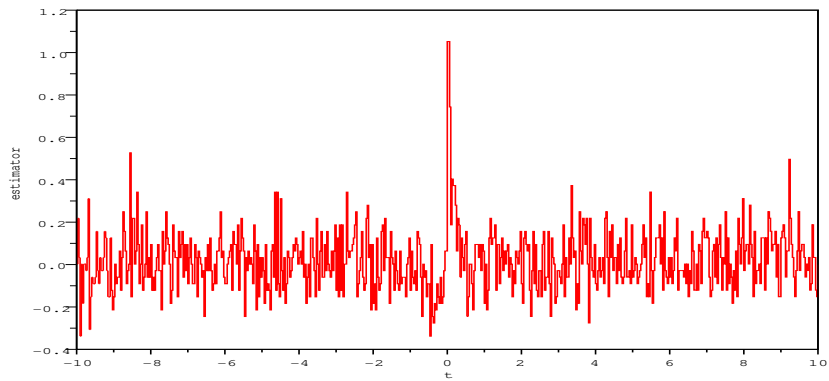}
  \caption{{\small\emph{Estimator, no thresholding, for E.\,coli data at the scale $1:1000$ (i.e.\,1 corresponds to 1000 bases), with parents=\texttt{tataat} and children=genes}}.}
  \label{fig6}
  \end{center}
\end{figure}

We apply the complete procedure proposed in Section 3.1 (with thresholding) and we obtain \ref{fig7}. The shape of this estimator explains how occurrences of genes are influenced by occurrences of \texttt{tataat}. We can draw following conclusions, that coincide with the ones we could expect:
\begin{itemize}
  \item The estimator $\tilde{h}(t)=0$ if $t\leq 0$ and $t\geq 500$. It means that for such $t$'s, gene occurrences seem to be uncorrelated of \texttt{tataat} occurrences.
  \item Conversely, if $t\in [0;500]$, $\tilde{h}(t)>0$, meaning that short distances are favored; smaller the distance, higher is the influence.
\end{itemize}

\begin{figure}[ht]
  \begin{center}
  \includegraphics[width=13cm,height=6cm]{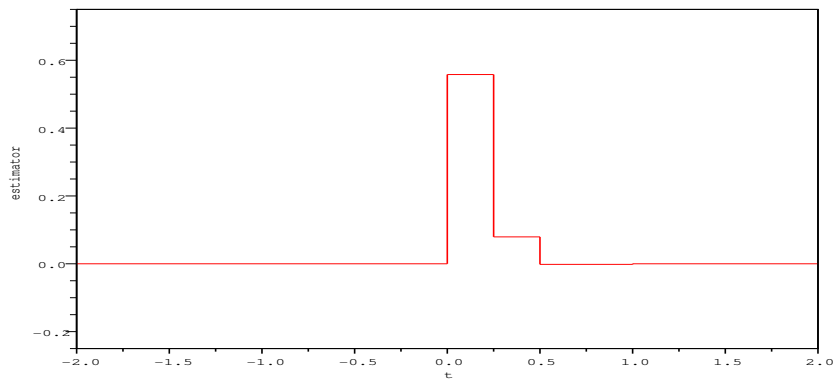}
  \caption{{\small\emph{Estimator $\tilde h$ defined in Section 3.1 for E.\,coli data at the scale $1:1000$, with parents=\texttt{tataat} and children=genes}}.}
  \label{fig7}
  \end{center}
\end{figure}

Then, we investigate the way genes influences the DNA motif \texttt{tataat} and so, in our model, the parents are the occurrences of genes and children are the occurrences of \texttt{tataat}.
\ref{fig8} gives the estimator $\tilde h$ defined in Section 3.1 (with $(\gamma,\delta)=(0.72,2.4)$). The shape of this estimator explains how occurrences of \texttt{tataat} are influenced by occurrences of genes. We can draw following conclusions, that is completely coherent with biological observations:
\begin{itemize}
  \item When $t\leq -500$ and $t\geq 1000$, $\tilde{h}(t)=0$. It means that for such $t$'s, \texttt{tataat} occurrences seem to be uncorrelated of gene occurrences.
  \item When $t\in [-500;0]$, $\tilde{h}(t)>0$, meaning that there is a preference having a word \texttt{tataat} just before the occurrence of a gene. It corresponds to the same conclusions drawn from \ref{fig7} (second point). The motif \texttt{tataat} is part of the most common promoter sites of \emph{E.\,coli} meaning that it should occur in front of the majority of the genes.
  \item When $t\in [0;1000]$, $\tilde{h}(t)<0$; occurrences of \texttt{tataat} are avoided for such distances $t$. Genes on the same strand do not usually overlap and they are about 1000 bases long in average: this fact can explain this conclusion.
\end{itemize}

\begin{figure}[ht]
  \begin{center}
  \includegraphics[width=13cm,height=10cm]{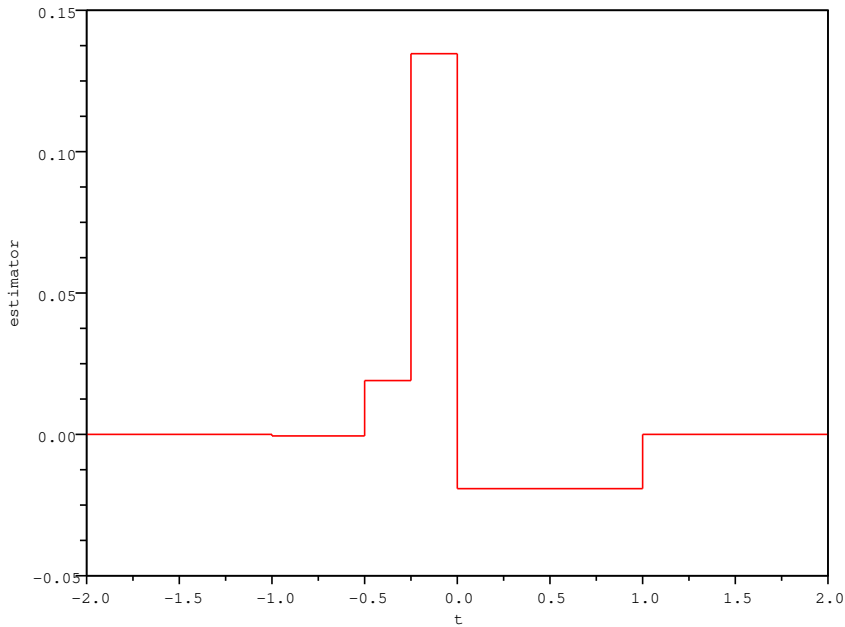}
  \caption{{\small\emph{Estimator $\tilde h$ defined in Section 3.1 for E.\,coli data at the scale $1:1000$, with parents=genes and children=\texttt{tataat}}}.}
  \label{fig8}
  \end{center}
\end{figure}

Finally, \ref{fig9} presents the results of the FADO procedure \cite{GS}
and \ref{fig10} presents the results of the Islands procedure of \cite{RBS}.
For the FADO procedure, we have forced the estimators to be piecewise constant to make the comparison easier.
Our results agree with the ones obtained by FADO and Islands.
But our method has advantage to point out that nothing significant happens after a certain distance (contrary to the FADO procedure), has advantage to treat interaction with another type of events (contrary to the Islands procedure) and has advantage to deal with the dependence on the past occurrences but also on the future occurrences (the function $h$ is supported in $\R_+$ for the two other procedures).
For algorithmic reason, a practical limitation of our method is that we only consider piecewise constant estimators (as for the Islands procedure), but it is enough to get a general trend on favored or avoided distances within a point process.

\begin{figure}[ht]
  \begin{center}
  \subfloat{\includegraphics[angle=90,scale=0.25]{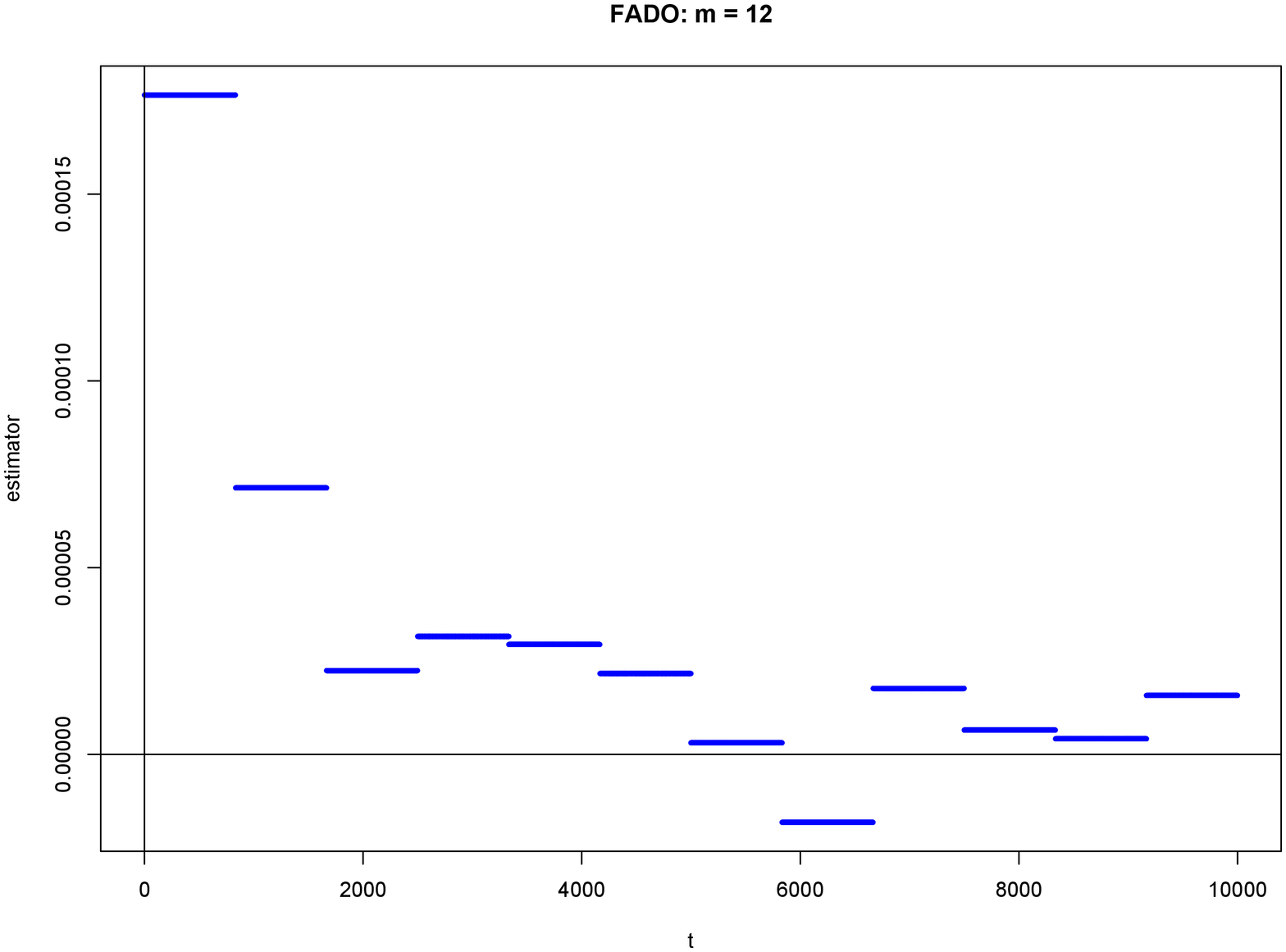}}
  \hspace{1cm}
  \subfloat{\includegraphics[angle=90,scale=0.25]{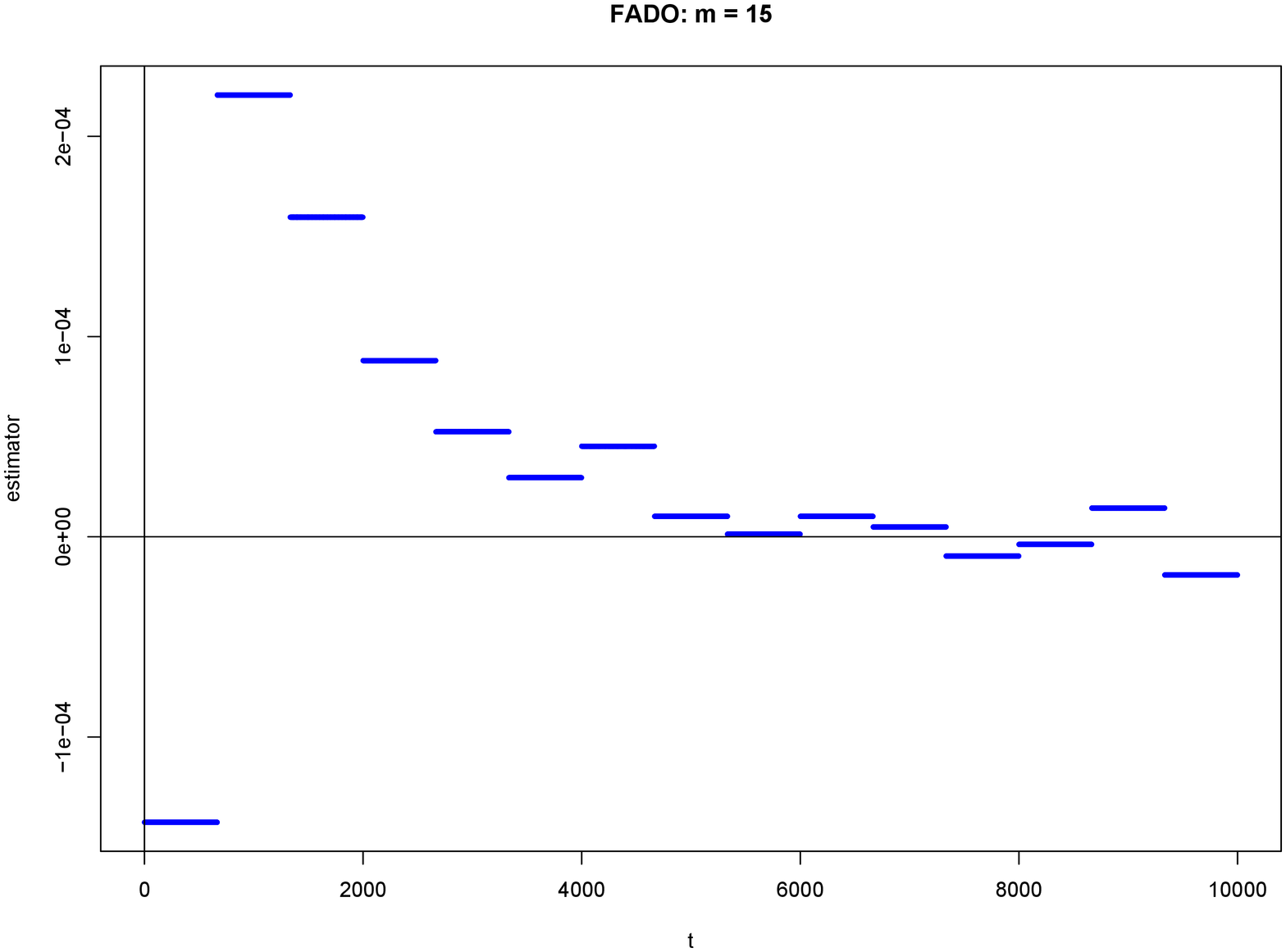}}
  \caption{{\small\emph{FADO estimators for both E.\,coli datasets: left: \texttt{tataat}; right: genes}}.}
  \label{fig9}
  \end{center}
\end{figure}

\begin{figure}[ht]
  \begin{center}
  \subfloat{\includegraphics[angle=90,scale=0.25]{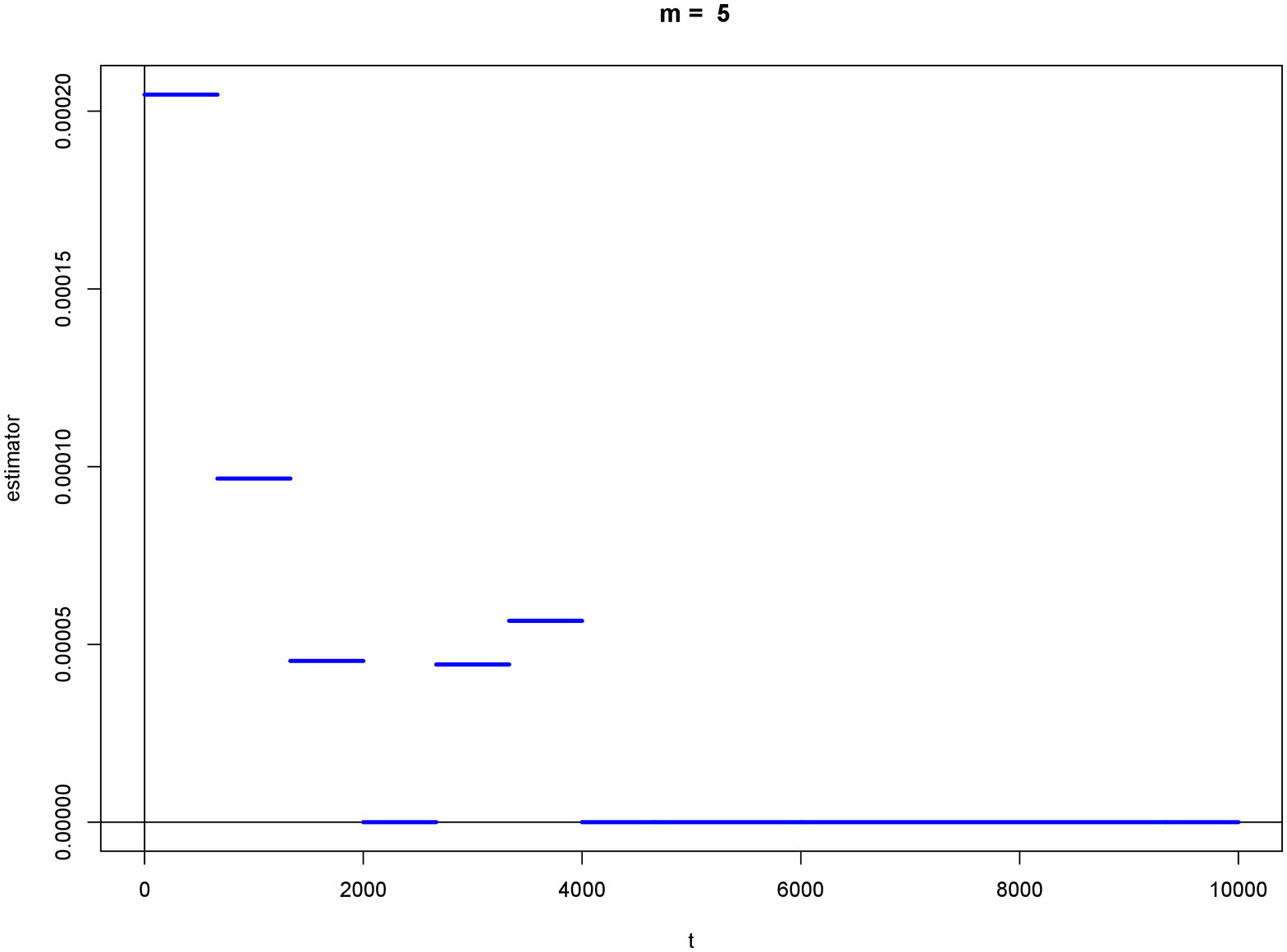}}
  \hspace{1cm}
  \subfloat{\includegraphics[angle=90,scale=0.25]{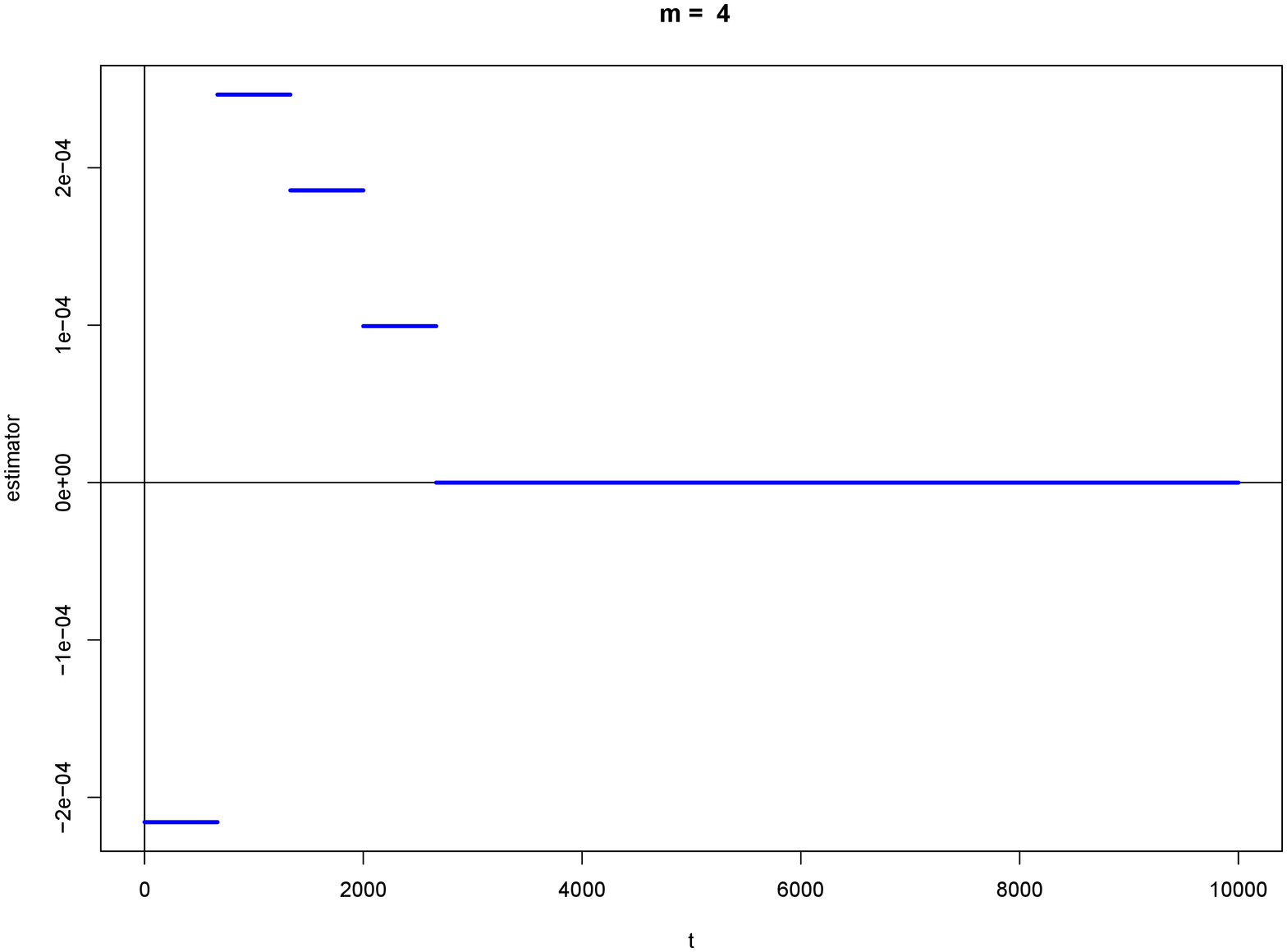}}
  \caption{{\small\emph{Islands estimators for both E.\,coli datasets: left: \texttt{tataat}; right: genes}}.}
  \label{fig10}
  \end{center}
\end{figure}

\section{Conclusion}

In our paper, we have investigated the dependencies between two given motifs.
A random thresholding procedure has been proposed in Section 2.2.
The general results of Section 2.3 have revealed the optimality of the procedure in the oracle and minimax setting.
Our theoretical results have been strengthened by simulations illustrating the robustness of our procedure, despite a calibration of parameters from a practical point of view that differs from the theoretical choice.
Section 4 has validated the procedure with a good detection of favored or avoided distances between occurrences of \texttt{tataat} and genes along the \emph{E.\,coli} genome.

Further extensions of our model could be investigated.
First, we could consider a more sophisticated model that takes into account the phenomenons of spontaneous apparition and self-excitation (as for the complete Hawkes model). But this model raises serious difficulties from the theoretical point of view. This is an exciting challenge to overcome them.
Secondly, we could extend our cascade algorithm to general wavelet bases and not only to Haar bases.
Finally, it is also relevant to study similar processes in the spatial framework and to connect them, for instance, to the Neymann-Scott process (see Section 6.3 of \cite{DVJ1}), which is a stimulating topic we wish to consider.

\medskip

\noindent \textbf{Acknowledgments:} The author wishes to thank Sophie Schbath for the two genomic data sets used in Section 4 and both her PhD advisors, Patricia Reynaud-Bouret and Vincent Rivoirard, for a wealth of smart advice and encouragement along this work.

\bibliographystyle{plain}
\bibliography{References}

\section{Appendix: Proof of \ref{thm1}}

In the sequel, the values of the constants $K, K', K_0, K_1, K_2, K_3, \ldots$ may change from line to line.
For the sake of clarity, the proofs are fully detailed in this appendix.

\subsection{A more general result}

We first give a general result stated and proved in \cite{RBR}.

\begin{thm}[Theorem 2.2 of \cite{RBR}]  \label{thmRBR}
To estimate a countable family $\beta = (\beta_{\lambda})_{\lambda\in\Lambda}$, such that $\|\beta\|_{\ell_{2}} < \infty$, we assume that a family of coefficient estimators $(\hat{\beta}_{\lambda})_{\lambda\in\Gamma}$, where $\Gamma$ is a known deterministic subset of $\Lambda$, and a family of possibly random thresholds $(\eta_{\lambda})_{\lambda\in\Gamma}$ are available and we consider the thresholding rule
\[\tilde{\beta} = \left( \hat{\beta}_{\lambda} \1_{|\hat{\beta}_{\lambda}|\geq\eta_{\lambda}} \1_{\lambda\in\Gamma} \right)_{\lambda\in\Lambda}.\]
Let $\varepsilon > 0$ be fixed. Assume that there exist a deterministic family $(H_{\lambda})_{\lambda\in\Gamma}$ and three constants $\kappa\in[0;1[$, $\omega\in[0;1]$ and $\zeta > 0$ (that may depend on $\varepsilon$ but not on $\lambda$) with the following properties:
\begin{itemize}
  \item[(A1)] For all $\lambda$ in $\Gamma$,
      \[\P\left(|\hat{\beta}_{\lambda}-\beta_{\lambda}| > \kappa\eta_{\lambda}\right) \leq \omega.\]
  \item[(A2)] There exist $1 < p,q < \infty$ with $\frac{1}{p}+\frac{1}{q}=1$ and a constant $R > 0$ such that for all $\lambda$ in $\Gamma$, \[\left[\E\left(|\hat{\beta}_{\lambda}-\beta_{\lambda}|^{2p}\right)\right]^{\frac{1}{p}} \leq R \max \left\{H_{\lambda} , H_{\lambda}^{\frac{1}{p}} \varepsilon^{\frac{1}{q}}\right\}.\]
  \item[(A3)] There exists a constant $\theta$ such that for all $\lambda$ in $\Gamma$ such that $H_{\lambda} < \theta\varepsilon$, \[\P\left(|\hat{\beta}_{\lambda}-\beta_{\lambda}| > \kappa\eta_{\lambda} , |\hat{\beta}_{\lambda}| > \eta_{\lambda}\right) \leq H_{\lambda}\zeta.\]
\end{itemize}
Then the estimator $\tilde{\beta}$ satisfies
\[\frac{1-\kappa^{2}}{1+\kappa^{2}} \E\left(\|\tilde{\beta}-\beta\|_{\ell_{2}}^{2}\right) \leq \E\left(\inf_{m\subset\Gamma} \left\{\frac{1+\kappa^{2}}{1-\kappa^{2}} \sum_{\lambda\not\in m} \beta_{\lambda}^{2} + \frac{1-\kappa^{2}}{\kappa^{2}} \sum_{\lambda\in m} (\hat{\beta}_{\lambda}-\beta_{\lambda})^{2} + \sum_{\lambda\in m} \eta_{\lambda}^{2}\right\}\right) + L D \sum_{\lambda\in\Gamma} H_{\lambda},\]
with $L D = \frac{R}{\kappa^{2}} \big((1+\theta^{-1/q}) \omega^{1/q} + (1+\theta^{1/q}) \varepsilon^{1/q}\zeta^{1/q}\big)$.
\end{thm}

Using the previous theorem, we establish the following result that we will prove in Section 6.4.

\begin{thm}  \label{thm3}
Let $n \geq 1$, $j_0\in\N^*$, $\gamma > 0$ and $\Delta$ defined by (\ref{def_Delta}) and (\ref{def_d}).
Then the estimator $\tilde{\beta}$ defined by (\ref{def_tildebeta}) in Section 2.2 satisfies
\[\E\left(\|\tilde{\beta}-\beta\|_{\ell_{2}}^{2}\right) \leq C_1 \inf_{m \subset \Gamma} \Bigg\{\sum_{\lambda \not \in m} \beta_{\lambda}^{2} + F(j_0,n,T) |m|\Bigg\} + C_2 R \big(e^{-\kappa_1j_0\gamma/2} + e^{-\kappa_2n\|h\|_1/2}\big) 2^{j_0} ,\]
where $C_1$ is a positive constant depending on $\gamma$, $\|h\|_1$, $\|h\|_{\infty}$, $\|\psi\|_1$, $\|\psi\|_2$ and $\|\psi\|_{\infty}$,
$C_2$ is a positive constant depending on the compact support of $h$ and the compact support of $\psi$,
\[F(j_0,n,T) = \frac{j_0}{n} + \frac{j_0^{3/2}2^{j_0/2}}{n^{3/2}} + \frac{j_0^4 2^{j_0}}{n^2} + \frac{j_0^2}{T} + \frac{j_0^{3/2}2^{j_0/2}}{nT^{1/2}} + \frac{j_0 n}{T^2},\]
\[R= C_R \ \Bigg\{ \frac{1}{n}  + \frac{2^{j_0/2}}{n^{3/2}} + \frac{2^{j_0/2}}{nT^{1/2}} + \frac{n}{T^{2}}  \Bigg\},\]
with $C_R$ a positive constant depending on $\|h\|_1$, $\|h\|_{\infty}$, the compact support of $\psi$, $\|\psi\|_1$, $\|\psi\|_2$ and $\|\psi\|_{\infty}$
and $\kappa_1$ and $\kappa_2$ are absolute constants in $]0;1[$.
\end{thm}

To obtain \ref{thm1}, we consider $n \geq 2$, we take $j_0$ the positive integer such that $2^{j_0} \leq n < 2^{j_0+1}$ and $\gamma > 2\log2$ in the previous theorem.
Therefore, we note that
\begin{align*}
F(j_0,n,T) & = \frac{j_0}{n} + \frac{j_0^{3/2}2^{j_0/2}}{n^{3/2}} + \frac{j_0^4 2^{j_0}}{n^2} + \frac{j_0^2}{T} + \frac{j_0^{3/2}2^{j_0/2}}{nT^{1/2}} + \frac{j_0 n}{T^2} \\
& \leq K \ \Bigg\{ \frac{\log{n}}{n} + \frac{(\log{n})^{3/2}n^{1/2}}{n^{3/2}} + \frac{(\log{n})^4 n}{n^2} + \frac{(\log{n})^2}{T} + \frac{(\log{n})^{3/2}n^{1/2}}{nT^{1/2}} + \frac{(\log{n}) n}{T^2} \Bigg\} \\
& \leq K \ \Bigg\{ \frac{\log{n}}{n} + \frac{(\log{n})^{3/2}}{n} + \frac{(\log{n})^4}{n} + \frac{(\log{n})^2}{T} + \frac{(\log{n})^{3/2}}{n^{1/2}T^{1/2}} + \frac{(\log{n}) n}{T^2} \Bigg\} \\
& \leq K \ \Bigg\{ \frac{(\log{n})^4}{n} + \frac{(\log{n}^2) n}{T^2} \Bigg\},
\end{align*}
by comparing all the terms of the right-hand side between them (for this, we distinguish the cases $n\leq T$ and $n\geq T$),
with $K$ an absolute positive constant (that changes from line to line) and
\[R \leq K' C_R \ \Bigg\{\frac{1}{n} + \frac{n}{T^{2}}\Bigg\},\]
with $K'$ an absolute positive constant.
Moreover, $e^{-\kappa_1j_0\gamma/2}2^{j_0}$ is bounded thanks to the choice of $\gamma$.

Finally, since
\[\|\tilde{h}-h\|_2^2 \leq K_0 \|\tilde{\beta}-\beta\|_{\ell_{2}}^{2},\]
with $K_0$ a positive constant depending only on the functions that generate the biorthogonal wavelet basis, we establish \ref{thm1}.

\subsection{Technical lemmas}

Before proving \ref{thm3}, we establish two lemmas which we will use throughout the proof.

\begin{lem} \label{lemINEQ} ~
\begin{itemize}
  \item[(a)] For any function $f$ in $\L_2(\R)$ and for all $t\in\R$, $\displaystyle \var_{\pi}(f(t-U)) \leq \frac{1}{T} \int_\R f^2(x) \,dx$, where \\ $\var_{\pi}(f(t-U))$ denotes the variance of $f(t-U)$ where $U \sim \pi$.
  \item[(b)] For any function $f$ in $\L_1(\R)$ and for all $t\in\R$, $\displaystyle \int_\R \E_{\pi}(f(t-U)) \,dt = \int_\R f(x) \,dx$.
  \item[(c)] For any nonnegative function $f$ in $\L_1(\R)$ and for all $t\in\R$, $\displaystyle \E_{\pi}(f(t-U)) \leq \frac{1}{T} \int_\R f(x) \,dx$.
\end{itemize}
\end{lem}

\begin{proof} ~
\begin{itemize}
  \item[(a)] Let $f\in\L_2(\R)$ and $t\in\R$.
      \[\var_{\pi}(f(t-U)) \leq \E_{\pi}(f^2(t-U)) = \frac{1}{T} \int_{0}^{T} f^2(t-u) \,du \leq \frac{1}{T} \int_\R f^2(x) \,dx.\]
  \item[(b)] Let $f\in\L_1(\R)$ and $t\in\R$.
      \[\int_\R \E_{\pi}(f(t-U)) \,dt = \E_{\pi}\left(\int_\R f(t-U) \,dt\right) = \E\left(\int_\R f(x) \,dx\right) = \int_\R f(x) \,dx.\]
  \item[(c)] Let $f\in\L_1(\R)$ such that $f \geq 0$ and $t\in\R$.
      \[\E_{\pi}(f(t-U)) = \frac{1}{T} \int_{0}^{T} f(t-u) \,du \leq \frac{1}{T} \int_\R f(x) \,dx.\]
\end{itemize}
\end{proof}

The next result is a Rosenthal type inequality for any Poisson process, that extends Lemma 6.2 of \cite{RBR}.

\begin{lem} \label{lemRIPP}
Let $p \geq 1$.
Consider a Poisson process $N$ on $(\mathbb{X},\mathcal{X})$ a measurable space, with a finite mean measure $\nu : \mathcal{X} \mapsto \R_+$ and a function $\varphi : \mathbb{X} \mapsto \R$ which belongs to $\L_{2p}(\nu)$.
We denote $\displaystyle \hat{\beta} = \int_{\mathbb{X}} \varphi(x) \,dN_x$ a natural estimator of $\displaystyle \beta = \int_{\mathbb{X}} \varphi(x) \,d\nu(x)$ that satisfies $\E(\hat{\beta}) = \beta$.
Then, there exists a positive constant $C(p)$ only depending on $p$ such that
\[\E(|\hat{\beta}-\beta|^{2p}) \leq C(p) \left(\int_\mathbb{X} |\varphi(x)|^{2p} \,d\nu(x) + \big(\var(\hat{\beta})\big)^p\right),\]
where $\displaystyle \var(\hat{\beta}) = \int_\mathbb{X} \varphi^2(x) \,d\nu(x)$.
\end{lem}

\begin{proof}
Let $p \geq 1$. Suppose $\|\varphi\|_{\infty} < +\infty$ first.
As a Poisson process is infinitely divisible, we can write: for any positive integer $k$, \[dN = \sum_{i=1}^{k} dN^i,\] where the $N^i$'s are mutually independent Poisson processes on $\mathbb{X}$ with mean measure $\nu / k$.
Hence, \[\hat{\beta}-\beta = \sum_{i=1}^{k} \int_{\mathbb{X}} \varphi(x) \,\big(dN^i_x - k^{-1}d\nu(x)\big) = \sum_{i=1}^{k} Y_i,\] where for any $i$, \[Y_i = \int_{\mathbb{X}} \varphi(x) \,\big(dN^i_x - k^{-1}d\nu(x)\big).\]
So the $Y_i$'s are i.i.d.\,centered variables, each of them has moments of order $2p$ and $2$.
We apply the classical Rosenthal's inequality (for instance, see Proposition 10.2 of \cite{HKPT}): there exists a positive constant $C(p)$ only depending on $p$ such that \[\E\left(\left|\sum_{i=1}^{k} Y_i\right|^{2p}\right) \leq C(p) \left(\sum_{i=1}^{k} \E(|Y_i|^{2p}) + \left(\sum_{i=1}^{k} \E(Y_i^2)\right)^p\right).\]
Now, we give an upper bound of the limit of $\displaystyle \E\left(\sum_{i=1}^{k} |Y_i|^\ell\right)$ for $\ell\in\{2p,2\}$ when $k \rightarrow \infty$.
Let us introduce \[\Omega_k = \big\{\forall i \in \{1,\ldots,k\}, N^i_{\mathbb{X}} \leq 1\big\},\] where $N^i_{\mathbb{X}}$ is the number of points of $N^i$ lying in $\mathbb{X}$.
Then,
\begin{align*}
\P(\Omega_k^c) & = \P(\exists i \in \{1,\ldots,k\}, N^i_{\mathbb{X}} \geq 2) \\
& \leq \sum_{i=1}^{k} \P(N^i_{\mathbb{X}} \geq 2) = k \sum_{j \geq 2} \frac{(\nu(\mathbb{X})/k)^j}{j!} e^{-\nu(\mathbb{X})/k} \\
& \leq k (\nu(\mathbb{X})/k)^2 = k^{-1}\nu(\mathbb{X})^2.
\end{align*}
On $\Omega_k$, if $N^i_{\mathbb{X}} = 0$ (so $\displaystyle \int_{\mathbb{X}} \varphi(x) \,dN^i_x = 0$), \[|Y_i|^\ell = O_k(k^{-\ell})\] and if $N^i_{\mathbb{X}} = 1$ (so $\displaystyle \int_{\mathbb{X}} \varphi(x) \,dN^i_x = \varphi(T)$, where $T$ is the point of the process $N^i$), \[|Y_i|^\ell = |\varphi(T)|^\ell + O_k(k^{-1}|\varphi(T)|^{\ell-1}).\]
Consequently,
\begin{equation}
\begin{split}
& \E\left(\sum_{i=1}^{k} |Y_i|^\ell\right) \\
& \leq \E\left[\1_{\Omega_k} \left(k O_k(k^{-\ell}) + \sum_{T \in N} \big[|\varphi(T)|^\ell + O_k(k^{-1}|\varphi(T)|^{\ell-1})\big]\right)\right] + \sqrt{\P(\Omega_k^c)} \sqrt{\E\left[\left(\sum_{i=1}^{k} |Y_i|^\ell\right)^2\right]}.
\end{split}  \label{pr-lemRIPP}
\end{equation}
But we have
\begin{align*}
\sum_{i=1}^{k} |Y_i|^\ell & \leq 2^{\ell-1} \sum_{i=1}^{k} \left[\left|\int_{\mathbb{X}} \varphi(x) \,dN^i_x\right|^\ell + \left(k^{-1}\int_{\mathbb{X}} |\varphi(x)| \,d\nu(x)\right)^\ell\right] \\
& \leq 2^{\ell-1} \left(\sum_{i=1}^{k} \|\varphi\|_\infty^\ell (N^i_\mathbb{X})^\ell + k\left(k^{-1}\int_{\mathbb{X}} |\varphi(x)| \,d\nu(x)\right)^\ell\right) \\
& \leq 2^{\ell-1} \left(\|\varphi\|_\infty^\ell N_\mathbb{X}^\ell + k\left(k^{-1}\int_{\mathbb{X}} |\varphi(x)| \,d\nu(x)\right)^\ell\right).
\end{align*}
Thus, when $k \rightarrow \infty$, the last term in (\ref{pr-lemRIPP}) converges to $0$ since a Poisson variable has moments of every order and
\[\limsup_{k \rightarrow \infty} \E\left(\sum_{i=1}^{k} |Y_i|^\ell\right) \leq \E\left(\int_\mathbb{X} |\varphi(x)|^\ell \,dN_x\right) = \int_\mathbb{X} |\varphi(x)|^\ell \,d\nu(x),\] which concludes the proof in the bounded case.

\noindent But for any function $\varphi$ such that $\int_\mathbb{X} |\varphi(x)|^{2p} \,d\nu(x) < +\infty$, the desired upper bound is finite and we get it by approximating $\varphi$ by, for instance, piecewise constant functions.
\end{proof}

\subsection{Proof of \ref{lemEST}}

Let $\lambda\in\Lambda$ be fixed. $G(\f)$, defined by (\ref{def_est}), is a measurable function of the observations and by considering the aggregated process (\ref{def_aggrproc}), we can write
\begin{align*}
G(\f) & = \int_\R \sum_{i=1}^{n} \left[\f(t-U_i) - \frac{n-1}{n} \E_{\pi}(\f(t-U))\right] \,dN_{t} \\
& = \int_\R \sum_{i=1}^{n} \f(t-U_i) \,dN_{t} - (n-1) \int_\R \E_{\pi}(\f(t-U)) \,dN_{t} \\
& = \sum_{1 \leq i,j \leq n} \int_\R \f(t-U_i) \,dN^{j}_{t} - \sum_{1 \leq i \neq j \leq n} \int_\R \E_{\pi}(\f(t-U)) \,dN^{j}_{t} \\
& = \sum_{i=1}^{n} \left[\int_\R \f(t-U_i) \,dN^{i}_{t} + \sum_{j \neq i} \int_\R \big[\f(t-U_i)-\E_{\pi}(\f(t-U))\big] \,dN^{j}_{t}\right].
\end{align*}

Now, we prove the first part of \ref{lemEST}. We have
\begin{align*}
& \E(G(\f)|U_1,\ldots,U_n) \\
& = \sum_{i=1}^{n} \left[\int_\R \f(t-U_i)h(t-U_i) \,dt + \sum_{j \neq i} \int_\R \big[\f(t-U_i)-\E_{\pi}(\f(t-U))\big] h(t-U_j) \,dt\right].
\end{align*}
Write $x=t-U_i$ in the first integral. Therefore, \[\E(G(\f)|U_1,\ldots,U_n) = n \int_\R \f(x)h(x) \,dx + W(\f),\] where
\[W(\f) = \sum_{1 \leq i \neq j \leq n} \int_\R \big[\f(t-U_i)-\E_{\pi}(\f(t-U))\big] h(t-U_j) \,dt.\]
Moreover,
\begin{align*}
\E(W(\f)) & = \sum_{1 \leq i \neq j \leq n} \int_\R \E\bigg(\big[\f(t-U_i)-\E_{\pi}(\f(t-U))\big] h(t-U_j)\bigg) \,dt \\
& = \sum_{1 \leq i \neq j \leq n} \int_\R \E\big[\f(t-U_i)-\E_{\pi}(\f(t-U))\big] \E(h(t-U_j)) \,dt \\
& = 0.
\end{align*}
Finally, \[\E(G(\f)) = n \int_\R \f(x)h(x) \,dx,\]
i.e.\,$\hat{\beta}_{\lambda}$ is an unbiased estimator for $\beta_{\lambda}$:
\[\E(\hat{\beta}_{\lambda}) = \E\left(\frac{G(\f)}{n}\right) = \int_\R \f(x)h(x) \,dx = \beta_{\lambda}.\]

\bigskip

It remains to control the variance of the estimator $\hat{\beta}_{\lambda}$.
\begin{align*}
\var(G(\f)) & = \E\left[\left(G(\f) - n \int_\R \f(x)h(x) \,dx\right)^2\right] \\
& = \E\big[(G(\f) - \E(G(\f)|U_1,\ldots,U_n) + W(\f))^2\big] \\
& = \E(V(\f)) + \E(W(\f)^2),
\end{align*}
where \[V(\f) = \var(G(\f)|U_1,\ldots,U_n).\]

We start by dealing with the first term by using technics for Poisson processes. We have
\begin{align*}
V(\f) & = \int_\R \left(\sum_{i=1}^{n} \left[\f(t-U_i) - \frac{n-1}{n} \E_{\pi}(\f(t-U))\right]\right)^2 \sum_{j=1}^{n} h(t-U_j) \,dt \\
& = \int_\R \sum_{j=1}^{n} \left(\f(t-U_j) + \sum_{i \neq j} \big[\f(t-U_i) - \E_{\pi}(\f(t-U))\big]\right)^2 h(t-U_j) \,dt \\
& = \sum_{j=1}^{n} \int_\R \f^2(t-U_j) h(t-U_j) \,dt \\
& \quad + 2\sum_{j=1}^{n} \int_\R \sum_{i \neq j} \big[\f(t-U_i) - \E_{\pi}(\f(t-U))\big] \f(t-U_j) h(t-U_j) \,dt \\
& \quad + \sum_{j=1}^{n} \int_\R \sum_{i \neq j} \sum_{k \neq j} \big[\f(t-U_i) - \E_{\pi}(\f(t-U))\big] \big[\f(t-U_k) - \E_{\pi}(\f(t-U))\big] h(t-U_j) \,dt.
\end{align*}
In the first integral, write $x=t-U_j$. So,
\begin{equation}
\begin{split}
V(\f) & = n \int_\R \f^2(x) h(x) \,dx + 2\sum_{j=1}^{n} \int_\R \sum_{i \neq j} \big[\f(t-U_i) - \E_{\pi}(\f(t-U))\big] \f(t-U_j) h(t-U_j) \,dt \\
& \quad + \sum_{j=1}^{n} \int_\R \sum_{i \neq j} \sum_{k \neq j} \big[\f(t-U_i) - \E_{\pi}(\f(t-U))\big] \big[\f(t-U_k) - \E_{\pi}(\f(t-U))\big] h(t-U_j) \,dt.
\end{split}  \label{def_V}
\end{equation}
Each term can be computed by taking the expectation conditionally to $U_j$ (in each sum) and we obtain
\begin{align}
\E(V(\f)) & = n \int_\R \f^2(x) h(x) \,dx + \sum_{j=1}^{n} \int_\R \sum_{i \neq j} \E\bigg(\big[\f(t-U_i) - \E_{\pi}(\f(t-U))\big]^2\bigg) \E(h(t-U_j)) \,dt \nonumber\\
& = n \int_\R \f^2(x) h(x) \,dx + n(n-1) \int_\R \var_{\pi}(\f(t-U)) \E_{\pi}(h(t-U)) \,dt.  \label{def_EV}
\end{align}
Then using equations (a) and (b) of \ref{lemINEQ}, we have
\begin{equation}
\E(V(\f)) \leq n \int_\R \f^2(x) h(x) \,dx + \frac{n(n-1)}{T} \int_\R \f^2(x) \,dx \int_\R h(x) \,dx.  \label{pr-lemEST_PP}
\end{equation}

\medskip

Now, we deal with the second term by using the $U$-statistics technics. However, $W(\f)$ is a $U$-statistics of order 2 but it is not degenerate. So we write
\begin{equation}
W(\f)=W_1(\f)+W_2(\f),  \label{def_W}
\end{equation}
with
\begin{align*}
W_1(\f) & = \sum_{1 \leq i \neq j \leq n} \int_\R \big[\f(t-U_i)-\E_{\pi}(\f(t-U))\big] \E_{\pi}(h(t-U)) \,dt \\
& = (n-1) \sum_{i=1}^{n} \int_\R \big[\f(t-U_i)-\E_{\pi}(\f(t-U))\big] \E_{\pi}(h(t-U)) \,dt
\end{align*}
and
\[W_2(\f) = \sum_{1 \leq i \neq j \leq n} g(U_i,U_j),\]
where \[g(U_i,U_j) = \int_\R \big[\f(t-U_i)-\E_{\pi}(\f(t-U))\big] \big[h(t-U_j)-\E_{\pi}(h(t-U))\big] \,dt.\]
$W_2(\f)$ is a degenerate $U$-statistics. It is easy to verify that \[\E(W(\f)^2) = \E(W_1(\f)^2) + \E(W_2(\f)^2).\]

First we compute $\E(W_1(\f)^2)$.
\begin{align}
\E(W_1(\f)^2) & = \var(W_1(\f)) \nonumber\\
& = n(n-1)^2 \var \left(\int_\R \big[\f(t-U_1)-\E_{\pi}(\f(t-U))\big] \E_{\pi}(h(t-U)) \,dt\right) \nonumber\\
& = n(n-1)^2 \var \left(\int_\R \f(t-U_1) \E_{\pi}(h(t-U)) \,dt\right) \nonumber\\
& \leq n(n-1)^2 \E \left[\left(\int_\R |\f(t-U_1)| \E_{\pi}(h(t-U)) \,dt\right)^2\right] \nonumber\\
& \leq \frac{n(n-1)^2}{T^2} \E\left[\left(\int_\R |\f(t-U_1)| \,dt\right)^2\right] \left(\int_\R h(x) \,dx\right)^2 \nonumber\\
& \leq \frac{n(n-1)^2}{T^2} \left(\int_\R |\f(x)| \,dx\right)^2 \left(\int_\R h(x) \,dx\right)^2,  \label{pr-lemEST_W1}
\end{align}
by applying inequality (c) of \ref{lemINEQ} with $f=h$.

It remains to compute $\E(W_2(\f)^2)$. It is easy to see that
\begin{align*}
\E(W_2(\f)^2) & = \sum_{1 \leq i \neq j \leq n} \E\big[g(U_i,U_j) (g(U_i,U_j)+g(U_j,U_i))\big] \leq 2\sum_{1 \leq i \neq j \leq n} \E\left[g(U_i,U_j)^2\right] \\
& \leq 2n(n-1) \E\left[\left(\int_\R \big[\f(t-U_1)-\E_{\pi}(\f(t-U))\big] \big[h(t-U_2)-\E_{\pi}(h(t-U))\big] \,dt\right)^2\right].
\end{align*}
We denote $\E_{(U,V) \sim \pi \otimes \pi}(g(U,V))$ the expectation of $g(U,V)$ where $U \sim \pi$ and $V \sim \pi$ are independent and $f^X(t) = f(t-X)$. Hence,
\begin{align*}
& \E(W_2(\f)^2) \\
& \leq 2n(n-1) \E_{(U,V) \sim \pi \otimes \pi}\left[\left(\int_\R \big[\f^U(t)-\E_{\pi}(\f^U(t))\big] \big[h^V(t)-\E_{\pi}(h^V(t))\big] \,dt\right)^2\right] \\
& \leq 2n(n-1) \E_{(U,V) \sim \pi \otimes \pi}\Bigg[\Bigg(\int_\R \f^U(t) h^V(t) \,dt - \E_{V \sim \pi}\left(\int_\R \f^U(t) h^V(t) \,dt\right) \\
& \hspace{11em} - \E_{U \sim \pi}\left(\int_\R \f^U(t) h^V(t) \,dt\right) + \E_{(U,V) \sim \pi \otimes \pi}\left(\int_\R \f^U(t) h^V(t) \,dt\right)\Bigg)^2\Bigg] \\
& \leq 2n(n-1) \Bigg\{ \E_{(U,V) \sim \pi \otimes \pi}\left[\left(\int_\R \f^U(t) h^V(t) \,dt\right)^2\right] - \E_{U \sim \pi}\left[\left(\E_{V \sim \pi}\left(\int_\R \f^U(t) h^V(t) \,dt\right)\right)^2\right] \\
& \hspace{6em} - \E_{V \sim \pi}\left[\left(\E_{U \sim \pi}\left(\int_\R \f^U(t) h^V(t) \,dt\right)\right)^2\right] + \left(\E_{(U,V) \sim \pi \otimes \pi}\left(\int_\R \f^U(t) h^V(t) \,dt\right)\right)^2 \Bigg\} \\
& \leq 2n(n-1) \Bigg\{ \E_{(U,V) \sim \pi \otimes \pi}\left[\left(\int_\R \f^U(t) h^V(t) \,dt\right)^2\right] + \left(\E_{(U,V) \sim \pi \otimes \pi}\left(\int_\R \f^U(t) h^V(t) \,dt\right)\right)^2 \Bigg\}.
\end{align*}
But,
\begin{align*}
\E_{(U,V) \sim \pi \otimes \pi}\left[\left(\int_\R \f^U(t) h^V(t) \,dt\right)^2\right] & \leq \E_{(U,V) \sim \pi \otimes \pi}\left(\int_\R (\f^U)^2(t) h^V(t) \,dt \int_\R h^V(t) \,dt\right) \\
& = \E_{(U,V) \sim \pi \otimes \pi}\left(\int_\R (\f^U)^2(t) h^V(t) \,dt\right) \int_\R h(x) \,dx \\
& = \int_\R \E_{\pi}\big((\f^U)^2(t)\big) \E_{\pi}(h^V(t)) \,dt \int_\R h(x) \,dx \\
& \leq \frac{1}{T} \int_\R \f^2(x) \,dx \left(\int_\R h(x) \,dx\right)^2
\end{align*}
and
\[\left|\E_{(U,V) \sim \pi \otimes \pi}\left(\int_\R \f^U(t) h^V(t) \,dt\right)\right| = \left|\int_\R \E_{\pi}(\f^U(t)) \E_{\pi}(h^V(t)) \,dt\right| \leq \frac{1}{T} \int_\R |\f(x)| \,dx \int_\R h(x) \,dx,\]
by using \ref{lemINEQ}.
So,
\begin{equation}
\E(W_2(\f)^2) \leq 2n(n-1) \Bigg\{ \frac{1}{T} \int_\R \f^2(x) \,dx \left(\int_\R h(x) \,dx\right)^2 + \frac{1}{T^2} \left(\int_\R |\f(x)| \,dx\right)^2 \left(\int_\R h(x) \,dx\right)^2 \Bigg\}.  \label{pr-lemEST_W2}
\end{equation}

\medskip

Finally, by combining inequalities (\ref{pr-lemEST_PP}), (\ref{pr-lemEST_W1}) and (\ref{pr-lemEST_W2}), we obtain the following control of the variance of the estimator $\hat{\beta}_{\lambda}$:
\begin{align*}
\var(\hat{\beta}_{\lambda}) & = \var\left(\frac{G(\f)}{n}\right) \\
& \leq \frac{1}{n} \int_\R \f^2(x) h(x) \,dx + \frac{1}{T} \int_\R \f^2(x) \,dx \int_\R h(x) \,dx \\
& \quad + \frac{n}{T^2} \left(\int_\R |\f(x)| \,dx\right)^2 \left(\int_\R h(x) \,dx\right)^2 + \frac{2}{T} \int_\R \f^2(x) \,dx \left(\int_\R h(x) \,dx\right)^2 \\
& \leq \frac{1}{n} \int_\R \f^2(x) h(x) \,dx + \frac{1}{T} \|\f\|_2^2 \|h\|_1 + \frac{n}{T^2} \|\f\|_1^2 \|h\|_1^2 + \frac{2}{T} \|\f\|_2^2 \|h\|_1^2.
\end{align*}
By using the properties of the biorthogonal wavelet bases considered in this paper, for any $\lambda=(j,k)$ in $\Lambda$, we have: $\|\f\|_1 \leq 2^{-j/2} \max(\sqrt{2}/2,\|\psi\|_1)$ and $\|\f\|_2 \leq \max(1,\|\psi\|_2)$, which allows us to get the purposed upper bound in \ref{lemEST}.

\subsection{Proof of \ref{thm3}}

In the sequel, we will consider: $n \geq 1$, $T \geq 1$ and $j_0 = O(n)$ and we will use following notations:
$M_{h,1} = \max(\|h\|_1,1)$,
$M_{h,\infty} = \max(\|h\|_\infty,1)$,
$M_{\psi,1} = \max(\|\psi\|_1,\sqrt{2}/2)$,
$M_{\psi,2} = \max(\|\psi\|_2,1)$ and
$M_{\psi,\infty} = \max(\|\psi\|_\infty,\sqrt{2})$
(so that, for any $\lambda=(j,k)\in\Lambda$, we have: $\|\f\|_1 \leq 2^{-j/2}M_{\psi,1}$, $\|\f\|_2 \leq M_{\psi,2}$ and $\|\f\|_\infty \leq 2^{j/2}M_{\psi,\infty}$).

We recall that $A$ and $M$ are positive real numbers such that $h$ and $\psi$ are compactly supported in $[-A;A]$ and in $[-M;M]$ respectively.

Now, to prove \ref{thm3}, we apply \ref{thmRBR} and for this purpose we have to verify Assumptions: (A1), (A2) and (A3).

\subsubsection{Proof of Assumption (A1)}

Let $\lambda \in \Gamma$ be fixed. Remember that conditionally to the $U_i$'s, the expression given in (\ref{def_aggrproc}) is a Poisson process. We apply Lemma 6.1 of \cite{RBR}: for any $\alpha > 0$, with probability larger than $1-2e^{-\alpha}$, conditionally to the $U_i$'s, we have \[\left|G(\f) - n \int_\R \f(x)h(x) \,dx - W(\f)\right| \leq \sqrt{2\alpha V(\f)} + \frac{\alpha}{3} B(\f),\]
where $W(\f)$ is defined by (\ref{def_W}), $V(\f) = \var(G(\f)|U_1,\ldots,U_n)$ and
\[B(\f) = \left\|\sum_{i=1}^{n} \left[\f(\cdot-U_i) - \frac{n-1}{n} \E_{\pi}(\f(\cdot-U))\right]\right\|_{\infty}.\]
Unlike $B(\f)$, $V(\f)$ is non-observable (it depends on the unknown function $h$). This is the reason why, by fixing $\alpha > 0$,  we estimate $V(\f)$ by
\[\widetilde{V}(\f) = \hat{V}(\f) + \sqrt{2\alpha\hat{V}(\f)B^2(\f)} + 3\alpha B^2(\f)\]
where
\[\hat{V}(\f) = \int_\R \left(\sum_{i=1}^{n} \left[\f(t-U_i) - \frac{n-1}{n} \E_{\pi}(\f(t-U))\right]\right)^2 \,dN_{t}.\]
Moreover, by Lemma 6.1 of \cite{RBR}, we have also: $\P(V(\f) \geq \widetilde{V}(\f)) \leq e^{-\alpha}$. So, with probability larger than $1-3e^{-\alpha}$,
\begin{equation}
\left|G(\f) - n \int_\R \f(x)h(x) \,dx\right| \leq \sqrt{2\alpha \widetilde{V}(\f)} + \frac{\alpha}{3} B(\f) + |W(\f)|.  \label{A1_PP}
\end{equation}

\medskip

We provide a control in probability of $W_1(\f)$.
\[W_1(\f) = (n-1) \sum_{i=1}^{n} \int_\R \big[\f(t-U_i)-\E_{\pi}(\f(t-U))\big] \E_{\pi}(h(t-U)) \,dt.\]
This is a sum of i.i.d.\,random variables. We apply Bernstein's inequality (for instance, see Proposition 2.9 of \cite{Mas}) to get that with probability larger than $1-2e^{-\alpha}$,
\[|W_1(\f)| \leq  \sqrt{2\alpha v(\f)} + \frac{\alpha}{3} b(\f),\]
with
\[v(\f) = \var(W_1(\f)) \leq \frac{n(n-1)^2}{T^2} \left(\int_\R |\f(x)| \,dx\right)^2 \left(\int_\R h(x) \,dx\right)^2\]
(see inequality (\ref{pr-lemEST_W1})) and
\begin{align*}
b(\f) & = (n-1) \sup_{u \in [0;T]}\left|\int_\R \big[\f(t-u)-\E_{\pi}(\f(t-U))\big] \E_{\pi}(h(t-U)) \,dt\right| \\
& \leq \frac{2(n-1)}{T} \int_\R |\f(x)| \,dx \int_\R h(x) \,dx,
\end{align*}
using equations (b) and (c) of \ref{lemINEQ}. Then, with probability larger than $1-2e^{-\alpha}$,
\begin{equation}
|W_1(\f)| \leq  \frac{\sqrt{2\alpha n}(n-1)}{T} \int_\R |\f(x)| \,dx \int_\R h(x) \,dx + \frac{2\alpha(n-1)}{3T} \int_\R |\f(x)| \,dx \int_\R h(x) \,dx.  \label{A1_W1}
\end{equation}

Now it remains to control $W_2(\f)$, with
\[W_2(\f) = \sum_{1 \leq i \neq j \leq n} g(U_i,U_j),\] where
\[g(U_i,U_j) = \int_\R \big[\f(t-U_i)-\E_{\pi}(\f(t-U))\big] \big[h(t-U_j)-\E_{\pi}(h(t-U))\big] \,dt.\]
This is a degenerate $U$-statistics of order 2, we can rewrite it as
\[W_2(\f) = \sum_{1 \leq j < i \leq n} \mathcal{G}(U_i,U_j),\] where \[\mathcal{G}(U_i,U_j) = g(U_i,U_j) + g(U_j,U_i).\]
We apply Theorem 3.4 of \cite{HRB} to $W_2$ and $-W_2$ (keeping the same notations of \cite{HRB}): for all $\varepsilon_0 > 0$ ($\varepsilon_0 = 1$ for instance), with probability larger than $1 - 2 \times 2.77 e^{-\alpha}$,
\begin{equation}
|W_2(\f)| \leq 2(1+\varepsilon_0)^{3/2}C\sqrt{\alpha} + \eta(\varepsilon_0)D\alpha + \beta(\varepsilon_0)B\alpha^{3/2} + \gamma(\varepsilon_0)A\alpha^2,  \label{A1_W2}
\end{equation}
where
\begin{itemize}
  \item $A = \|\mathcal{G}\|_{\infty}$ and by applying equality (b) of \ref{lemINEQ} with $f=h$, we easily have
      \begin{equation}
      A \leq 8 \|\f\|_{\infty} \int_\R h(x) \,dx,  \label{A1_W2-A}
      \end{equation}
  \item $C^2 = \E(W_2(\f)^2)$ and with (\ref{pr-lemEST_W2}), we have
      \begin{equation}
      C^2 \leq 2n(n-1) \Bigg\{ \frac{1}{T} \int_\R \f^2(x) \,dx \left(\int_\R h(x) \,dx\right)^2 + \frac{1}{T^2} \left(\int_\R
      |\f(x)|\,dx\right)^2 \left(\int_\R h(x) \,dx\right)^2 \Bigg\},  \label{A1_W2-C}
      \end{equation}
  \item $\displaystyle D = \sup\left\{\E\left(\sum_{1 \leq j < i \leq n} \mathcal{G}(U_i,U_j)a_i(U_i)b_j(U_j)\right) : \E\left(\sum_{i=2}^{n} a_i(U_i)^2\right) \leq 1, \E\left(\sum_{j=1}^{n-1} b_j(U_j)^2\right) \leq 1\right\}$,
      \begin{align*}
      D & = \sup\Bigg\{\E\left(\sum_{i=2}^{n} \sum_{j=1}^{i-1} g(U_i,U_j)a_i(U_i)b_j(U_j) + \sum_{j=1}^{n-1} \sum_{i=j+1}^{n}
      g(U_j,U_i)a_i(U_i)b_j(U_j)\right) : \\
      & \hspace{15em} \E\left(\sum_{i=2}^{n} a_i(U_i)^2\right) \leq 1, \E\left(\sum_{j=1}^{n-1} b_j(U_j)^2\right) \leq 1\Bigg\}.
      \end{align*}
      But, with the conditions on the $a_i$'s and the $b_j$'s, we have:
      \begin{align*}
      & \E\left(\sum_{i=2}^{n} \sum_{j=1}^{i-1} g(U_i,U_j)a_i(U_i)b_j(U_j)\right) \\
      & \leq \int_\R \E\left(\sum_{i=2}^{n} \big|\f(t-U_i)-\E_{\pi}(\f(t-U))\big| |a_i(U_i)| \right) \E\left(\sum_{j=1}^{n-1}
      \big|h(t-U_j)-\E_{\pi}(h(t-U))\big| |b_j(U_j)|\right) \,dt \\
      & \leq \int_\R \sqrt{(n-1)\var_{\pi}(\f(t-U))} \ \E\left(\sum_{j=1}^{n-1} \big|h(t-U_j)-\E_{\pi}(h(t-U))\big| |b_j(U_j)|\right) \,dt \\
      & \leq \sqrt{\frac{n-1}{T} \int_\R \f^2(x) \,dx} \ \E\left(\sum_{j=1}^{n-1} \int_\R \big|h(t-U_j)-\E_{\pi}(h(t-U))\big| |b_j(U_j)|
      \,dt\right) \\
      & \leq \sqrt{\frac{n-1}{T} \int_\R \f^2(x) \,dx} \ \E\left(2\sum_{j=1}^{n-1} |b_j(U_j)| \int_\R h(x) \,dx\right) \\
      & \leq 2(n-1) \sqrt{\frac{1}{T} \int_\R \f^2(x) \,dx} \int_\R h(x) \,dx,
      \end{align*}
      using equations (a) and (b) of \ref{lemINEQ}. Inverting the $a_i$'s and the $b_j$'s, the same computations apply to the second term and we
      obtain
      \begin{equation}
      D \leq 4(n-1) \sqrt{\frac{1}{T} \int_\R \f^2(x) \,dx} \int_\R h(x) \,dx,  \label{A1_W2-D}
      \end{equation}
  \item $\displaystyle B^2 = \max\left\{\sup_{u,i}\left(\sum_{j=1}^{i-1} \E(\mathcal{G}(u,U_j)^2)\right),\sup_{u,j}\left(\sum_{i=j+1}^{n} \E(\mathcal{G}(U_i,u)^2)\right)\right\}$ and since $\mathcal{G}$ is symmetric, we have:
      \begin{align*}
      B^2 & = \sup_{u,i}\left(\sum_{j=1}^{i-1} \E(\mathcal{G}(u,U_j)^2)\right) \\
      & = \sup_{u}\left(\sum_{j=1}^{n-1} \E(\mathcal{G}(u,U_j)^2)\right) \\
      & \leq 2\sup_{u}\left(\sum_{j=1}^{n-1} \big[\E(g(u,U_j)^2)+\E(g(U_j,u)^2)\big]\right).
      \end{align*}
      But,
      \begin{align*}
      & \E(g(u,U_j)^2) \\
      & = \E\left[\left(\int_\R \big[\f(t-u)-\E_{\pi}(\f(t-U))\big] \big[h(t-U_j)-\E_{\pi}(h(t-U))\big] \,dt\right)^2\right] \\
      & \leq \E\left[\int_\R \big[\f(t-u)-\E_{\pi}(\f(t-U))\big]^2 \big|h(t-U_j)-\E_{\pi}(h(t-U))\big| \,dt \int_\R
      \big|h(t-U_j)-\E_{\pi}(h(t-U))\big| \,dt\right] \\
      & \leq 2 \E\left[\int_\R \big[\f(t-u)-\E_{\pi}(\f(t-U))\big]^2 \big|h(t-U_j)-\E_{\pi}(h(t-U))\big| \,dt\right] \int_\R h(x) \,dx \\
      & \leq \frac{4}{T} \int_\R \big[\f(t-u)-\E_{\pi}(\f(t-U))\big]^2 \,dt \left(\int_\R h(x) \,dx\right)^2 \\
      & \leq \frac{16}{T} \int_\R \f^2(x) \,dx \left(\int_\R h(x) \,dx\right)^2
      \end{align*}
      and in the same way
      \begin{align*}
      & \E(g(U_j,u)^2) \\
      & = \E\left[\left(\int_\R \big[\f(t-U_j)-\E_{\pi}(\f(t-U))\big] \big[h(t-u)-\E_{\pi}(h(t-U))\big] \,dt\right)^2\right] \\
      & \leq \E\left[\int_\R \big[\f(t-U_j)-\E_{\pi}(\f(t-U))\big]^2 \big|h(t-u)-\E_{\pi}(h(t-U))\big| \,dt \int_\R
      \big|h(t-u)-\E_{\pi}(h(t-U))\big| \,dt\right] \\
      & \leq 2 \E\left[\int_\R \big[\f(t-U_j)-\E_{\pi}(\f(t-U))\big]^2 \big|h(t-u)-\E_{\pi}(h(t-U))\big| \,dt\right] \int_\R h(x) \,dx \\
      & \leq \frac{2}{T} \int_\R \big|h(t-u)-\E_{\pi}(h(t-U))\big| \,dt \int_\R \f^2(x) \,dx \int_\R h(x) \,dx \\
      & \leq \frac{4}{T} \int_\R \f^2(x) \,dx \left(\int_\R h(x) \,dx\right)^2,
      \end{align*}
      by using \ref{lemINEQ}. Hence,
      \begin{equation}
      B^2 \leq \frac{40(n-1)}{T} \int_\R \f^2(x) \,dx \left(\int_\R h(x) \,dx\right)^2.  \label{A1_W2-B}
      \end{equation}
\end{itemize}

\medskip

Finally, by inequalities (\ref{A1_PP}), (\ref{A1_W1}) and (\ref{A1_W2}) combined with (\ref{A1_W2-A}), (\ref{A1_W2-B}), (\ref{A1_W2-C}) and (\ref{A1_W2-D}), we obtain: for any $\varepsilon_0 > 0$, with probability larger than $1-(5+2\times2.77)e^{-\alpha}$,
\begin{align*}
|\hat{\beta}_{\lambda} - \beta_{\lambda}| & \leq \sqrt{2\alpha \widetilde{V}\left(\frac{\f}{n}\right)} + \frac{\alpha}{3} B\left(\frac{\f}{n}\right) \\
& \quad + \Bigg\{ \frac{\sqrt{2\alpha n}}{T} \|\f\|_1 + \frac{2\alpha}{3T} \|\f\|_1 + 2(1+\varepsilon_0)^{3/2}\sqrt{2\alpha} \sqrt{\frac{1}{T} \|\f\|_2^2 + \frac{1}{T^2} \|\f\|_1^2} \\
& \hspace{5em} + 4\eta(\varepsilon_0)\alpha \sqrt{\frac{1}{T} \|\f\|_2^2} + \beta(\varepsilon_0)\alpha^{3/2}  \sqrt{\frac{40}{nT} \|\f\|_2^2} + \frac{8}{n}\gamma(\varepsilon_0)\alpha^2  \|\f\|_{\infty} \Bigg\} \|h\|_1 \\
& \leq \sqrt{2 \alpha \widetilde{V}\left(\frac{\f}{n}\right)} + \frac{\alpha}{3} B\left(\frac{\f}{n}\right) \\
& \quad + \Bigg\{ 2 M_{\psi,1} \frac{\sqrt{\alpha n}}{T}  + 2\sqrt{2} M_{\psi,1} \frac{\alpha}{3T} + 2\sqrt{2}(1+\varepsilon_0)^{3/2} M_{\psi,2} \sqrt{\frac{\alpha}{T}} + 4(1+\varepsilon_0)^{3/2} M_{\psi,1} \frac{\sqrt{\alpha}}{T} \\
& \hspace{5em} + 4\eta(\varepsilon_0) M_{\psi,2} \frac{\alpha}{\sqrt{T}} + \sqrt{40}\beta(\varepsilon_0) M_{\psi,2} \frac{\alpha^{3/2}}{\sqrt{nT}} + 8\gamma(\varepsilon_0) M_{\psi,\infty} \frac{2^{j_0/2}\alpha^2}{n} \Bigg\} \|h\|_1,
\end{align*}
because $2^{-j/2} \leq \sqrt{2}$ if $-1 \leq j \leq j_0$. We denote $b$ the quantity between braces above.

This upper bound depends on $h$ (via $\|h\|_1$) and this potential threshold could not be used for applications because $h$ is unknown. So we overestimate $\|h\|_1$ by $\frac{(1+\varepsilon_0)N_{\R}}{n}$ and we have a threshold that does not depend on $h$. So, for any value of $\kappa \in ]0;1[$, by fixing $\alpha=\kappa^2j_0\gamma$ with $\gamma>0$, we define for all $\lambda$ in $\Gamma$,
\[\eta_\lambda(\gamma,\Delta) = \sqrt{2j_0\gamma \widetilde{V}\left(\frac{\f}{n}\right)} + \frac{j_0\gamma}{3} B\left(\frac{\f}{n}\right) + \Delta \frac{N_\R}{n},\]
where
\begin{align*}
\Delta & = (1+\varepsilon_0) \Bigg\{ 2 M_{\psi,1} \frac{\sqrt{j_0\gamma n}}{T} + 2\sqrt{2} M_{\psi,1} \frac{j_0\gamma}{3T} + 2\sqrt{2}(1+\varepsilon_0)^{3/2} M_{\psi,2} \sqrt{\frac{j_0\gamma}{T}} + 4(1+\varepsilon_0)^{3/2} M_{\psi,1} \frac{\sqrt{j_0\gamma}}{T} \\
& \hspace{5em} + 4\eta(\varepsilon_0) M_{\psi,2} \frac{j_0\gamma}{\sqrt{T}} + \sqrt{40}\beta(\varepsilon_0) M_{\psi,2} \frac{j_0^{3/2}\gamma^{3/2}}{\sqrt{nT}} + 8\gamma(\varepsilon_0) M_{\psi,\infty} \frac{2^{j_0/2}j_0^2\gamma^2}{n} \Bigg\}.
\end{align*}
Thus, for all $\lambda$ in $\Gamma$,
\begin{align*}
& \P\big(|\hat{\beta}_{\lambda} - \beta_{\lambda}| > \kappa\eta_\lambda(\gamma,\Delta)\big) \\
& \leq \P\left(|\hat{\beta}_{\lambda} - \beta_{\lambda}| > \sqrt{2\alpha \widetilde{V}\left(\frac{\f}{n}\right)} + \frac{\alpha}{3} B\left(\frac{\f}{n}\right) + b \frac{(1+\varepsilon_0)N_{\R}}{n} \,,\, \frac{(1+\varepsilon_0)N_{\R}}{n} > \|h\|_1\right) \\
& \quad + \P\left(|\hat{\beta}_{\lambda} - \beta_{\lambda}| > \sqrt{2\alpha \widetilde{V}\left(\frac{\f}{n}\right)} + \frac{\alpha}{3} B\left(\frac{\f}{n}\right) + b \frac{(1+\varepsilon_0)N_{\R}}{n} \,,\, \frac{(1+\varepsilon_0)N_{\R}}{n} \leq \|h\|_1\right) \\
& \leq \P\left(|\hat{\beta}_{\lambda} - \beta_{\lambda}| > \sqrt{2\alpha \widetilde{V}\left(\frac{\f}{n}\right)} + \frac{\alpha}{3} B\left(\frac{\f}{n}\right) + b \|h\|_1\right) + \P\left(\frac{(1+\varepsilon_0)N_{\R}}{n} \leq \|h\|_1\right) \\
& \leq (5+2\times2.77)e^{-\alpha} + \P\left(\frac{(1+\varepsilon_0)N_{\R}}{n} \leq \|h\|_1\right),
\end{align*}
with
\begin{align*}
\P\left(\frac{(1+\varepsilon_0)N_{\R}}{n} \leq \|h\|_1\right) & = \P\left(N_{\R} - n\|h\|_1 \leq - \frac{\varepsilon_0 n\|h\|_1}{1+\varepsilon_0}\right) \\
& \leq \exp{(-g(\varepsilon_0) n\|h\|_1)},
\end{align*}
using Proposition 7 of \cite{Rey2} with $g(\varepsilon_0) = \frac{1}{1+\varepsilon_0} \left(\log{\frac{1}{1+\varepsilon_0}}-1\right) + 1$.

Therefore, Assumption (A1) is true if we take $\omega=(5+2\times2.77)e^{-\kappa^2j_0\gamma} + \exp{(-g(\varepsilon_0) n\|h\|_1)}$,
with $\gamma>0$ and $\varepsilon_0>0$. Furthermore, the threshold (\ref{def_threshold}) that lies at the heart of the paper is achieved by rewriting $\Delta$ by grouping the constants into one:
\begin{equation}
\Delta = d(\gamma,\|\psi\|_1,\|\psi\|_2,\|\psi\|_\infty) \left\{\frac{j_0^2 2^{j_0/2}}{n} + \frac{j_0}{\sqrt{T}} + \frac{\sqrt{j_0n}}{T}\right\}  \label{def_Delta}
\end{equation}
with
\begin{equation}
\begin{split}
& d(\gamma,\|\psi\|_1,\|\psi\|_2,\|\psi\|_\infty) \\
& = (1+\varepsilon_0) \Bigg\{ 2 \sqrt{\gamma} M_{\psi,1} + \frac{2\sqrt{2}}{3} \gamma M_{\psi,1} + 2\sqrt{2}(1+\varepsilon_0)^{3/2} \sqrt{\gamma} M_{\psi,2} + 4(1+\varepsilon_0)^{3/2} \sqrt{\gamma} M_{\psi,1} \\
& \hspace{14em} + 4\eta(\varepsilon_0) \gamma M_{\psi,2} + \sqrt{40}\beta(\varepsilon_0) \gamma^{3/2} M_{\psi,2} + 8\gamma(\varepsilon_0) \gamma^2 M_{\psi,\infty} \Bigg\},
\end{split}  \label{def_d}
\end{equation}
where $\beta(\varepsilon_0)$, $\gamma(\varepsilon_0)$ and $\eta(\varepsilon_0)$ are defined in \cite{HRB} with $\varepsilon_0=1$.

\subsubsection{Proof of Assumption (A2)}

Let $\lambda \in \Gamma$ be fixed. For any $p \geq 1$,
\begin{align}
\E(|\hat{\beta}_{\lambda} - \beta_{\lambda}|^{2p}) & = \E\left(\left|G\left(\frac{\f}{n}\right) - \int_\R \f(x)h(x) \,dx\right|^{2p}\right) \nonumber\\
& = \frac{1}{n^{2p}} \E\left(\big|G(\f) - \E(G(\f)|U_1,\ldots,U_n) + W(\f)\big|^{2p}\right) \nonumber\\
& \leq \frac{2^{2p-1}}{n^{2p}} \big[\E(|G(\f) - \E(G(\f)|U_1,\ldots,U_n)|^{2p}) + \E(|W(\f)|^{2p})\big].  \label{A2_general}
\end{align}
Now, let us give an upper bound of each term of the right-hand side of the previous inequality.

We first study the first term of (\ref{A2_general}). We have: \[\E(|G(\f) - \E(G(\f)|U_1,\ldots,U_n)|^{2p}) = \E\big[\E(|G(\f) - \E(G(\f)|U_1,\ldots,U_n)|^{2p}) \,|\, U_1,\ldots,U_n)\big]\] and conditionally to the $U_i$'s, $N$ is a Poisson process. We apply \ref{lemRIPP}: for any $p \geq 1$, there exists a positive constant $C(p)$ only depending on $p$ such that
\begin{equation}
\begin{split}
& \E(|G(\f) - \E(G(\f)|U_1,\ldots,U_n)|^{2p}) \,|\, U_1,\ldots,U_n) \\
& \leq C(p) \left(\int_\R \left|\sum_{i=1}^{n} \left[\f(t-U_i) - \frac{n-1}{n} \E_{\pi}(\f(t-U))\right]\right|^{2p} \sum_{j=1}^{n} h(t-U_j) \,dt + V(\f)^p\right).
\end{split}  \label{A2_PP}
\end{equation}

\medskip

On the one hand, we provide a control in expectation of the first term of (\ref{A2_PP}). We have:
\begin{align*}
& \E\left(\int_\R \left|\sum_{i=1}^{n} \left[\f(t-U_i) - \frac{n-1}{n} \E_{\pi}(\f(t-U))\right]\right|^{2p} \sum_{j=1}^{n} h(t-U_j) \,dt\right) \\
& = \E\left(\int_\R \sum_{j=1}^{n} \left|\f(t-U_j) + \sum_{i\neq j} \big[\f(t-U_i) - \E_{\pi}(\f(t-U))\big]\right|^{2p} h(t-U_j) \,dt\right) \\
& \leq 2^{2p-1} \Bigg[\E\left(\int_\R \sum_{j=1}^{n} |\f(t-U_j)|^{2p} h(t-U_j) \,dt\right) \\
& \hspace{4em} + \E\left(\int_\R \sum_{j=1}^{n} \left|\sum_{i\neq j} \big[\f(t-U_i) - \E_{\pi}(\f(t-U))\big]\right|^{2p} h(t-U_j) \,dt\right)\Bigg],
\end{align*}
with
\[\E\left(\int_\R \sum_{j=1}^{n} |\f(t-U_j)|^{2p} h(t-U_j) \,dt\right) = n \int_\R |\f(x)|^{2p} h(x) \,dx \leq n \int_\R \f^2(x) \,dx \|\f\|_{\infty}^{2p-2} \|h\|_{\infty}\]
and
\begin{align*}
& \E\left(\int_\R \sum_{j=1}^{n} \left|\sum_{i\neq j} \big[\f(t-U_i) - \E_{\pi}(\f(t-U))\big]\right|^{2p} h(t-U_j) \,dt\right) \\
& = \sum_{j=1}^{n} \int_\R \E\left(\left|\sum_{i\neq j} \big[\f(t-U_i) - \E_{\pi}(\f(t-U))\big]\right|^{2p}\right) \E(h(t-U_j)) \,dt \\
& = n \int_\R \E\left(\left|\sum_{i=1}^{n-1} \big[\f(t-U_i) - \E_{\pi}(\f(t-U))\big]\right|^{2p}\right) \E_{\pi}(h(t-U)) \,dt.
\end{align*}
By applying Rosenthal's inequality, there exists a positive constant $C(p)$ only depending on $p$ such that
\begin{align*}
& \E\left(\left|\sum_{i=1}^{n-1} \big[\f(t-U_i) - \E_{\pi}(\f(t-U))\big]\right|^{2p}\right) \\
& \leq C(p) \left((n-1) \E\left(\big|\f(t-U_1) - \E_{\pi}(\f(t-U))\big|^{2p}\right) + (n-1)^p \big[\var_{\pi}(\f(t-U))\big]^p\right).
\end{align*}
But,
\begin{align*}
\E\left(\big|\f(t-U_1) - \E_{\pi}(\f(t-U))\big|^{2p}\right) & \leq K \ \left( \E_{\pi}(|\f(t-U)|^{2p}) + |\E_{\pi}(\f(t-U))|^{2p} \right) \\
& \leq K \ \left( \frac{1}{T} \int_0^T |\f(t-u)|^{2p} \,du + \left|\frac{1}{T} \int_0^T \f(t-u) \,du\right|^{2p} \right) \\
& \leq K \ \left( \frac{1}{T} \int_\R \f^2(x) \,dx \ \|\f\|_{\infty}^{2p-2} + \frac{1}{T^{2p}} \left(\int_\R |\f(x)| \,dx\right)^{2p} \right),
\end{align*}
with $K$ a positive constant only depending on $p$ and using inequality (a) of \ref{lemINEQ},
\[\var_{\pi}(\f(t-U)) \leq \frac{1}{T} \int_\R \f^2(x) \,dx.\]
Thus,
\begin{align*}
& \E\left(\int_\R \sum_{j=1}^{n} \left|\sum_{i\neq j} \big[\f(t-U_i) - \E_{\pi}(\f(t-U))\big]\right|^{2p} h(t-U_j) \,dt\right) \\
& \leq n \ K \ \left(\frac{n-1}{T} \|\f\|_2^2 \|\f\|_{\infty}^{2p-2} + \frac{n-1}{T^{2p}} \|\f\|_1^{2p} + \frac{(n-1)^p}{T^p} \|\f\|_2^{2p}\right) \int_\R \E_{\pi}(h(t-U)) \,dt \\
& \leq K \ \left(\frac{n^2}{T} \|\f\|_2^2 \|\f\|_{\infty}^{2p-2} + \frac{n^2}{T^{2p}} \|\f\|_1^{2p} + \frac{n^{p+1}}{T^p} \|\f\|_2^{2p}\right) \|h\|_1,
\end{align*}
using equation (b) of \ref{lemINEQ}.

Therefore, we have the following control of the first term of (\ref{A2_PP})
\begin{equation}
\begin{split}
& \E\left(\int_\R \left|\sum_{i=1}^{n} \left[\f(t-U_i) - \frac{n-1}{n} \E_{\pi}(\f(t-U))\right]\right|^{2p} \sum_{j=1}^{n} h(t-U_j) \,dt\right) \\
& \leq K \ \Bigg[n \|\f\|_2^2 \|\f\|_{\infty}^{2p-2} \|h\|_{\infty} + \left(\frac{n^2}{T} \|\f\|_2^2 \|\f\|_{\infty}^{2p-2} + \frac{n^2}{T^{2p}} \|\f\|_1^{2p} + \frac{n^{p+1}}{T^p} \|\f\|_2^{2p}\right) \|h\|_1\Bigg],
\end{split}  \label{A2_PP-1st}
\end{equation}

\medskip

Now let us provide a control in expectation of the second term of (\ref{A2_PP}), i.e\,$V(\f)^p$. First, we recall that $V(\f) = \var(G(\f)|U_1,\ldots,U_n)$ and we remark that $\E(V(\f)^p) \leq \big[\E(V(\f)^{2p})\big]^{1/2}$ (using the Cauchy-Schwarz inequality). So, we focus on the moments of $V(\f)$ of any order $m \geq 2$.

Let $m \geq 2$. According to the expression (\ref{def_V}) of $V(\f)$, we have:
\begin{align*}
& V(\f) \\
& = n \int_\R \f^2(x) h(x) \,dx + 2\sum_{j=1}^{n} \int_\R \sum_{i\neq j} \big[\f(t-U_i) - \E_{\pi}(\f(t-U))\big] \f(t-U_j) h(t-U_j) \,dt \\
& \quad + \sum_{j=1}^{n} \int_\R \sum_{i\neq j} \sum_{k\neq j} \big[\f(t-U_i) - \E_{\pi}(\f(t-U))\big] \big[\f(t-U_k) - \E_{\pi}(\f(t-U))\big] h(t-U_j) \,dt \\
& = n \int_\R \f^2(x) h(x) \,dx + 2\sum_{j=1}^{n} \int_\R \sum_{i\neq j} \big[\f(t-U_i) - \E_{\pi}(\f(t-U))\big] \E_{\pi}(\f(t-U) h(t-U))) \,dt \\
& \quad + 2\sum_{j=1}^{n} \int_\R \sum_{i\neq j} \big[\f(t-U_i) - \E_{\pi}(\f(t-U))\big] \big[\f(t-U_j) h(t-U_j) - \E_{\pi}(\f(t-U) h(t-U))\big] \,dt \\
& \quad + \sum_{j=1}^{n} \int_\R \sum_{i\neq j} \sum_{k\neq j} \big[\f(t-U_i) - \E_{\pi}(\f(t-U))\big] \big[\f(t-U_k) - \E_{\pi}(\f(t-U))\big] \E_{\pi}(h(t-U)) \,dt \\
& \quad + \sum_{j=1}^{n} \int_\R \sum_{i\neq j} \sum_{k\neq j} \big[\f(t-U_i) - \E_{\pi}(\f(t-U))\big] \big[\f(t-U_k) - \E_{\pi}(\f(t-U))\big] \\
& \hspace{26em} \times \big[h(t-U_j) - \E_{\pi}(h(t-U))\big] \,dt.
\end{align*}
This formula provides a decomposition of $V(\f)$ in a sum of degenerate $U$-statistics of order 0, 1, 2 and 3. Indeed
\[V(\f) = \mathcal{W}_0(\f) + \mathcal{W}_1(\f) + \mathcal{W}_2(\f) + \mathcal{W}_3(\f),\]
with $\mathcal{W}_i(\f)$ is a degenerate $U$-statistic of order $i$ defined as follows:
\begin{align*}
\mathcal{W}_3(\f) & = \sum_{1 \leq i \neq j \neq k \leq n} \int_\R \big[\f(t-U_i) - \E_{\pi}(\f(t-U))\big] \big[\f(t-U_k) - \E_{\pi}(\f(t-U))\big] \\
& \hspace{25em} \times \big[h(t-U_j) - \E_{\pi}(h(t-U))\big] \,dt,
\end{align*}
\begin{align*}
& \mathcal{W}_2(\f) \\
& = 2\sum_{1 \leq i \neq j \leq n} \int_\R \big[\f(t-U_i) - \E_{\pi}(\f(t-U))\big] \big[\f(t-U_j) h(t-U_j) - \E_{\pi}(\f(t-U) h(t-U))\big] \,dt \\
& \quad + (n-2) \sum_{1 \leq i \neq k \leq n} \int_\R \big[\f(t-U_i) - \E_{\pi}(\f(t-U))\big] \big[\f(t-U_k) - \E_{\pi}(\f(t-U))\big] \E_{\pi}(h(t-U)) \,dt \\
& \quad  + \sum_{1 \leq i \neq j \leq n} \int_\R \bigg[\big[\f(t-U_i) - \E_{\pi}(\f(t-U))\big]^2 - \var_{\pi}(\f(t-U))\bigg] \big[h(t-U_j) - \E_{\pi}(h(t-U))\big] \,dt,
\end{align*}
\begin{align*}
\mathcal{W}_1(\f) & = 2(n-1) \sum_{i=1}^{n} \int_\R \big[\f(t-U_i) - \E_{\pi}(\f(t-U))\big] \E_{\pi}(\f(t-U) h(t-U)) \,dt \\
& \quad + (n-1) \sum_{i=1}^{n} \int_\R \bigg[\big[\f(t-U_i) - \E_{\pi}(\f(t-U))\big]^2 - \var_{\pi}(\f(t-U))\bigg] \E_{\pi}(h(t-U)) \,dt \\
& \quad + (n-1) \sum_{j=1}^{n} \int_\R \var_{\pi}(\f(t-U)) \big[h(t-U_j) - \E_{\pi}(h(t-U))\big] \,dt
\end{align*}
and
\begin{align}
\mathcal{W}_0(\f) & = n \int_\R \f^2(x) h(x) \,dx + n(n-1) \int_\R \var_{\pi}(\f(t-U)) \E_{\pi}(h(t-U)) \,dt \nonumber\\
& = \E(V(\f)) \leq n \int_\R \f^2(x) h(x) \,dx + \frac{n(n-1)}{T} \int_\R \f^2(x) \,dx \int_\R h(x) \,dx,  \label{A2_PP-2nd:W0}
\end{align}
by using (\ref{def_EV}) and (\ref{pr-lemEST_PP}).

First, we are interested in the moments of $\mathcal{W}_1(\f)$ that we write: \[\mathcal{W}_1(\f) = \mathcal{W}_{1,1}(\f) + \mathcal{W}_{1,2}(\f) + \mathcal{W}_{1,3}(\f),\]
with:
\begin{itemize}
  \item $\displaystyle \mathcal{W}_{1,1}(\f) = 2(n-1) \sum_{i=1}^{n} \int_\R \big[\f(t-U_i) - \E_{\pi}(\f(t-U))\big] \E_{\pi}(\f(t-U) h(t-U)) \,dt$.
\end{itemize}
We have:
\begin{align*}
& \E(|\mathcal{W}_{1,1}(\f)|^{m}) \\
& = 2^m (n-1)^m \E\left(\left|\sum_{i=1}^{n} \int_\R \big[\f(t-U_i) - \E_{\pi}(\f(t-U))\big] \E_{\pi}(\f(t-U) h(t-U)) \,dt\right|^m\right) \\
& \leq 2^m (n-1)^m \times C(m) \Bigg(n\E\left(\left|\int_\R \big[\f(t-U_1) - \E_{\pi}(\f(t-U))\big] \E_{\pi}(\f(t-U) h(t-U)) \,dt\right|^m\right) \\
& \hspace{8em} + n^{m/2}\left[\var\left(\int_\R \big[\f(t-U_1) - \E_{\pi}(\f(t-U))\big] \E_{\pi}(\f(t-U) h(t-U)) \,dt\right)\right]^{m/2}\Bigg),
\end{align*}
using Rosenthal's inequality, where $C(m)$ is a positive constant only depending on $m$. But, applying \ref{lemINEQ},
\begin{align*}
& \E\left(\left|\int_\R \big[\f(t-U_1) - \E_{\pi}(\f(t-U))\big] \E_{\pi}(\f(t-U) h(t-U)) \,dt\right|^m\right) \\
& \leq \frac{2^m}{T^m} \left(\int_\R |\f(x)| \,dx\right)^m \left(\int_\R |\f(x)|h(x) \,dx\right)^m
\end{align*}
and
\begin{align*}
& \var\left(\int_\R \big[\f(t-U_1) - \E_{\pi}(\f(t-U))\big] \E_{\pi}(\f(t-U) h(t-U)) \,dt\right) \\
& = \var\left(\int_\R \f(t-U_1) \E_{\pi}(\f(t-U) h(t-U)) \,dt\right) \\
& \leq \E \left[\left(\int_\R |\f(t-U_1)| \E_{\pi}(|\f(t-U)| h(t-U)) \,dt\right)^2\right] \\
& \leq \frac{1}{T^2} \left(\int_\R |\f(x)| \,dx\right)^2 \left(\int_\R |\f(x)|h(x) \,dx\right)^2.
\end{align*}
So,
\begin{align}
\E(|\mathcal{W}_{1,1}(\f)|^{m}) & \leq 2^m (n-1)^m \times C(m) \Bigg(\frac{2^m n}{T^m} \left(\int_\R |\f(x)| \,dx\right)^m \left(\int_\R |\f(x)|h(x) \,dx\right)^m \nonumber\\
& \hspace{11em} + \frac{n^{m/2}}{T^m} \left(\int_\R |\f(x)| \,dx\right)^m \left(\int_\R |\f(x)|h(x) \,dx\right)^m\Bigg) \nonumber\\
& \leq K_{1,1} \ \frac{n^{3m/2}}{T^m} \left(\int_\R |\f(x)| \,dx\right)^m \left(\int_\R |\f(x)|h(x) \,dx\right)^m \nonumber\\
& \leq K_{1,1} \ \frac{n^{3m/2}}{T^m} \left(\int_\R |\f(x)| \,dx\right)^{2m} \|h\|_{\infty}^m,  \label{A2_PP-2nd:W11}
\end{align}
with $K_{1,1}$ a positive constant only depending on $m$.
\begin{itemize}
  \item $\displaystyle \mathcal{W}_{1,2}(\f) = (n-1) \sum_{i=1}^{n} \int_\R \bigg[\big[\f(t-U_i) - \E_{\pi}(\f(t-U))\big]^2 - \var_{\pi}(\f(t-U))\bigg] \E_{\pi}(h(t-U)) \,dt$.
\end{itemize}
We have:
\begin{align*}
& \E(|\mathcal{W}_{1,2}(\f)|^{m}) \\
& = (n-1)^m \E\left(\left|\sum_{i=1}^{n} \int_\R \bigg[\big[\f(t-U_i) - \E_{\pi}(\f(t-U))\big]^2 - \var_{\pi}(\f(t-U))\bigg] \E_{\pi}(h(t-U)) \,dt\right|^m\right) \\
& \leq n^m \times C(m) \Bigg(n\E\left(\left|\int_\R \bigg[\big[\f(t-U_1) - \E_{\pi}(\f(t-U))\big]^2 - \var_{\pi}(\f(t-U))\bigg] \E_{\pi}(h(t-U)) \,dt\right|^m\right) \\
& \hspace{2em} + n^{m/2}\left[\var\left(\int_\R \bigg[\big[\f(t-U_1) - \E_{\pi}(\f(t-U))\big]^2 - \var_{\pi}(\f(t-U))\bigg] \E_{\pi}(h(t-U)) \,dt\right)\right]^{m/2}\Bigg),
\end{align*}
using Rosenthal's inequality, where $C(m)$ is a positive constant only depending on $m$. But, applying \ref{lemINEQ},
\begin{align*}
& \E\left(\left|\int_\R \bigg[\big[\f(t-U_1) - \E_{\pi}(\f(t-U))\big]^2 - \var_{\pi}(\f(t-U))\bigg] \E_{\pi}(h(t-U)) \,dt\right|^m\right) \\
& \leq K_{1,2} \ \left(\int_\R \f^2(x) \,dx\right)^m \left(\int_\R h(x) \,dx\right)^m,
\end{align*}
with $K_{1,2}$ a positive constant only depending on $m$ and
\begin{align*}
& \var\left(\int_\R \bigg[\big[\f(t-U_1) - \E_{\pi}(\f(t-U))\big]^2 - \var_{\pi}(\f(t-U))\bigg] \E_{\pi}(h(t-U)) \,dt\right) \\
& = \var\left(\int_\R \big[\f(t-U_1) - \E_{\pi}(\f(t-U))\big]^2 \E_{\pi}(h(t-U)) \,dt\right) \\
& \leq \E \left[\left(\int_\R \big[\f(t-U_1) - \E_{\pi}(\f(t-U))\big]^2 \E_{\pi}(h(t-U)) \,dt\right)^2\right] \\
& \leq \frac{1}{T^2} \E \left[\left(\int_\R \big[\f(t-U_1) - \E_{\pi}(\f(t-U))\big]^2 \,dt\right)^2\right] \left(\int_\R h(x) \,dx\right)^2 \\
& \leq K_{1,2} \ \frac{1}{T^2} \left(\int_\R \f^2(x) \,dx\right)^2 \left(\int_\R h(x) \,dx\right)^2
\end{align*}
So,
\begin{align}
& \E(|\mathcal{W}_{1,2}(\f)|^{m}) \nonumber\\
& \leq K_{1,2} \ n^m \Bigg(\frac{n}{T^m} \left(\int_\R \f^2(x) \,dx\right)^m \left(\int_\R h(x) \,dx\right)^m + \frac{n^{m/2}}{T^m} \left(\int_\R \f^2(x) \,dx\right)^m \left(\int_\R h(x) \,dx\right)^m\Bigg) \nonumber\\
& \leq K_{1,2} \ \frac{n^{3m/2}}{T^m} \left(\int_\R \f^2(x) \,dx\right)^m \left(\int_\R h(x) \,dx\right)^m.  \label{A2_PP-2nd:W12}
\end{align}
\begin{itemize}
  \item $\displaystyle \mathcal{W}_{1,3}(\f) = (n-1) \sum_{j=1}^{n} \int_\R \var_{\pi}(\f(t-U)) \big[h(t-U_j) - \E_{\pi}(h(t-U))\big] \,dt$.
\end{itemize}
We have:
\begin{align*}
& \E(|\mathcal{W}_{1,3}(\f)|^{m}) \\
& = (n-1)^m \E\left(\left|\sum_{j=1}^{n} \int_\R \var_{\pi}(\f(t-U)) \big[h(t-U_j) - \E_{\pi}(h(t-U))\big] \,dt\right|^m\right) \\
& \leq (n-1)^m \times C(m) \Bigg(n\E\left(\left|\int_\R \var_{\pi}(\f(t-U)) \big[h(t-U_1) - \E_{\pi}(h(t-U))\big] \,dt\right|^m\right) \\
& \hspace{10em} + n^{m/2}\left[\var\left(\int_\R \var_{\pi}(\f(t-U)) \big[h(t-U_1) - \E_{\pi}(h(t-U))\big] \,dt\right)\right]^{m/2}\Bigg),
\end{align*}
using Rosenthal's inequality, where $C(m)$ is a positive constant only depending on $m$. But, applying \ref{lemINEQ},
\[\E\left(\left|\int_\R \var_{\pi}(\f(t-U)) \big[h(t-U_1) - \E_{\pi}(h(t-U))\big] \,dt\right|^m\right) \leq \frac{2^m}{T^m} \left(\int_\R \f^2(x) \,dx\right)^m \left(\int_\R h(x) \,dx\right)^m,\]
and
\begin{align*}
& \var\left(\int_\R \var_{\pi}(\f(t-U)) \big[h(t-U_1) - \E_{\pi}(h(t-U))\big] \,dt\right) \\
& = \var\left(\int_\R \var_{\pi}(\f(t-U)) h(t-U_1) \,dt\right) \\
& \leq \E \left[\left(\int_\R \var_{\pi}(\f(t-U)) h(t-U_1) \,dt\right)^2\right] \\
& \leq \frac{1}{T^2} \left(\int_\R \f^2(x) \,dx\right)^2 \E \left[\left(\int_\R h(t-U_1) \,dt\right)^2\right] \\
& \leq \frac{1}{T^2} \left(\int_\R \f^2(x) \,dx\right)^2 \left(\int_\R h(x) \,dx\right)^2.
\end{align*}
So,
\begin{align}
\E(|\mathcal{W}_{1,3}(\f)|^{m}) & \leq (n-1)^m \times C(m) \Bigg(\frac{2^m n}{T^m} \left(\int_\R \f^2(x) \,dx\right)^m \left(\int_\R h(x) \,dx\right)^m \nonumber\\
& \hspace{10em} + \frac{n^{m/2}}{T^m} \left(\int_\R \f^2(x) \,dx\right)^m \left(\int_\R h(x) \,dx\right)^m\Bigg) \nonumber\\
& \leq K_{1,3} \ \frac{n^{3m/2}}{T^m} \left(\int_\R \f^2(x) \,dx\right)^m \left(\int_\R h(x) \,dx\right)^m,  \label{A2_PP-2nd:W13}
\end{align}
with $K_{1,3}$ a positive constant only depending on $m$.

Next we deal with the moments of $\mathcal{W}_2(\f)$ that we write: \[\mathcal{W}_2(\f) = \mathcal{W}_{2,1}(\f) + \mathcal{W}_{2,2}(\f) + \mathcal{W}_{2,3}(\f),\]
with:
\begin{itemize}
  \item $\displaystyle \mathcal{W}_{2,1}(\f) = 2\sum_{1 \leq i \neq j \leq n} \int_\R [\f(t-U_i) - \E_{\pi}(\f(t-U))] \big[\f(t-U_j) h(t-U_j) - \E_{\pi}(\f(t-U) h(t-U))\big] \,dt$.
\end{itemize}
We want to use Theorem 8.1.6 of \cite{dPG} (a moment inequality for $U$-statistics using decoupling) so we write:
\[\mathcal{W}_{2,1}(\f) = 2 \sum_{1 \leq i \neq j \leq n} f(U_i,U_j),\]
where \[f(U_i,U_j) = \int_\R \big[\f(t-U_i) - \E_{\pi}(\f(t-U))\big] \big[\f(t-U_j) h(t-U_j) - \E_{\pi}(\f(t-U) h(t-U))\big] \,dt.\]
There exists a positive constant $C_{2,m}$ depending on $m$ only such that
\begin{align*}
& \E\left(\left|\sum_{1 \leq i \neq j \leq n} f(U_i,U_j)\right|^m\right) \\
& \leq C_{2,m} n^m \E(|f(U_1,U_2)|^m) \\
& \leq C_{2,m} n^m \E\left(\left|\int_\R [\f(t-U_1) - \E_{\pi}(\f(t-U))] \big[\f(t-U_2) h(t-U_2) - \E_{\pi}(\f(t-U) h(t-U))\big] \,dt\right|^m\right) \\
& \leq C_{2,m} n^m \E\left[\left(\int_\R \big[\f(t-U_1) - \E_{\pi}(\f(t-U))\big]^2 \,dt\right)^{m/2}\right] \\
& \quad \times \E\left[\left(\int_\R \big[\f(t-U_2) h(t-U_2) - \E_{\pi}(\f(t-U) h(t-U))\big]^2 \,dt\right)^{m/2}\right] \\
& \leq K_{2,1} \ n^m  \left(\int_\R \f^2(x) \,dx\right)^{m/2} \left(\int_\R \f^2(x)h^2(x) \,dx\right)^{m/2} \\
& \leq K_{2,1} \ n^m \left(\int_\R \f^2(x) \,dx\right)^{m} \|h\|_{\infty}^m,
\end{align*}
by applying \ref{lemINEQ} and setting $K_{2,1}$ a positive constant only depending on $m$.
So,
\begin{equation}
\E(|\mathcal{W}_{2,1}(f)|^{m}) \leq K_{2,1} \ n^m \left(\int_\R \f^2(x) \,dx\right)^{m} \|h\|_{\infty}^m.  \label{A2_PP-2nd:W21}
\end{equation}
\begin{itemize}
  \item We may write: $\displaystyle \mathcal{W}_{2,2}(\f) = (n-2) \sum_{1 \leq i \neq k \leq n} f(U_i,U_k)$,
      where \[f(U_i,U_k) = \int_\R \big[\f(t-U_i) - \E_{\pi}(\f(t-U))\big] \big[\f(t-U_k) - \E_{\pi}(\f(t-U))\big] \E_{\pi}(h(t-U)) \,dt.\]
\end{itemize}
We use Theorem 8.1.6 of \cite{dPG}: there exists a positive constant $C_{2,m}$ depending on $m$ only such that
\begin{align*}
& \E\left(\left|\sum_{1 \leq i \neq k \leq n} f(U_i,U_k)\right|^m\right) \\
& \leq C_{2,m} n^m \E(|f(U_1,U_2)|^m) \\
& \leq C_{2,m} n^m \E\left(\left|\int_\R \big[\f(t-U_1) - \E_{\pi}(\f(t-U))\big] \big[\f(t-U_2) - \E_{\pi}(\f(t-U))\big] \E_{\pi}(h(t-U)) \,dt\right|^m\right) \\
& \leq K_{2,2} \ \frac{n^m}{T^m} \left\{\E\left[\left(\int_\R \big[\f(t-U_1) - \E_{\pi}(\f(t-U))\big]^2 \,dt\right)^{m/2}\right]\right\}^2 \left(\int_\R h(x) \,dx\right)^{m} \\
& \leq K_{2,2} \ \frac{n^m}{T^m} \left(\int_\R \f^2(x) \,dx\right)^{m} \left(\int_\R h(x) \,dx\right)^{m},
\end{align*}
by applying \ref{lemINEQ} and setting $K_{2,2}$ a positive constant only depending on $m$.
So,
\begin{equation}
\E(|\mathcal{W}_{2,2}(f)|^{m}) \leq K_{2,2} \ \frac{n^{2m}}{T^m} \left(\int_\R \f^2(x) \,dx\right)^{m} \left(\int_\R h(x) \,dx\right)^{m}.  \label{A2_PP-2nd:W22}
\end{equation}
\begin{itemize}
  \item We may write: $\displaystyle \mathcal{W}_{2,3}(\f) = \sum_{1 \leq i \neq j \leq n} f(U_i,U_j)$,
      where \[f(U_i,U_j)=\int_\R\bigg[\big[\f(t-U_i)-\E_{\pi}(\f(t-U))\big]^2-\var_{\pi}(\f(t-U))\bigg]\big[h(t-U_j)-\E_{\pi}(h(t-U))\big]\,dt.\]
\end{itemize}
We use Theorem 8.1.6 of \cite{dPG}: there exists a positive constant $C_{2,m}$ depending on $m$ only such that
\begin{align*}
& \E\left(\left|\sum_{1 \leq i \neq j \leq n} f(U_i,U_j)\right|^m\right) \\
& \leq C_{2,m} n^m \E(|f(U_1,U_2)|^m) \\
& \leq C_{2,m} n^m \E\left(\left|\int_\R [[\f(t-U_1) - \E_{\pi}(\f(t-U))]^2 - \var_{\pi}(\f(t-U))] [h(t-U_2) - \E_{\pi}(h(t-U))] \,dt\right|^m\right) \\
& \leq K_{2,3} \ n^m \E\left[\left(\int_\R \left|\big[\f(t-U_1) - \E_{\pi}(\f(t-U))\big]^2 - \var_{\pi}(\f(t-U))\right| \,dt\right)^m\right] \|h\|_{\infty}^m \\
& \leq K_{2,3} \ n^m \left(\int_\R \f^2(x) \,dx\right)^{m} \|h\|_{\infty}^m,
\end{align*}
by applying \ref{lemINEQ} and setting $K_{2,3}$ a positive constant only depending on $m$.
So,
\begin{equation}
\E(|\mathcal{W}_{2,3}(\f)|^{m}) \leq K_{2,3} \ n^m \left(\int_\R \f^2(x) \,dx\right)^{m} \|h\|_{\infty}^m.  \label{A2_PP-2nd:W23}
\end{equation}
And finally, we focus on the moments of $\mathcal{W}_3(\f)$ that we write: $\displaystyle \mathcal{W}_{3}(\f) = \sum_{1 \leq i \neq j \neq k \leq n} f(U_i,U_j,U_k)$,
where \[f(U_i,U_j,U_k) = \int_\R \big[\f(t-U_i) - \E_{\pi}(\f(t-U))\big] \big[\f(t-U_k) - \E_{\pi}(\f(t-U))\big] \big[h(t-U_j) - \E_{\pi}(h(t-U))\big] \,dt.\]
We use Theorem 8.1.6 of \cite{dPG}: there exists a positive constant $C_{3,m}$ depending on $m$ only such that
\begin{align*}
& \E\left(\left|\sum_{1 \leq i \neq j \neq k \leq n} f(U_i,U_j,U_k)\right|^m\right) \\
& \leq C_{3,m} n^{3m/2} \E(|f(U_1,U_2,U_3)|^m) \\
& \leq C_{3,m} n^{3m/2} \E\Bigg(\Bigg|\int_\R \big[\f(t-U_1) - \E_{\pi}(\f(t-U))\big] \big[\f(t-U_2) - \E_{\pi}(\f(t-U))\big] \\
& \hspace{27em} \times \big[h(t-U_3) - \E_{\pi}(h(t-U))\big] \,dt\Bigg|^m\Bigg) \\
& \leq K_3 \ n^{3m/2} \E\left[\left(\int_\R \Big|\big[\f(t-U_1) - \E_{\pi}(\f(t-U))\big] \big[\f(t-U_2) - \E_{\pi}(\f(t-U))\big]\Big| \,dt\right)^m\right] \|h\|_{\infty}^m,
\end{align*}
by applying \ref{lemINEQ} and setting $K_3$ a positive constant only depending on $m$.
Furthermore, using the support properties of the biorthogonal wavelet bases considered in this paper, we have
\begin{align*}
& \E\left[\left(\int_\R \Big|\big[\f(t-U_1) - \E_{\pi}(\f(t-U))\big] \big[\f(t-U_2) - \E_{\pi}(\f(t-U))\big]\Big| \,dt\right)^m\right] \\
& = \E\Bigg[\Bigg(\int_\R \Big|\f(t-U_1)\f(t-U_2) - \f(t-U_1)\E_{\pi}(\f(t-U)) - \f(t-U_2)\E_{\pi}(\f(t-U)) \\
& \hspace{32em} + \big[\E_{\pi}(\f(t-U))\big]^2\Big| \,dt\Bigg)^m\Bigg] \\
& \leq K_3 \ \Bigg\{\E\Bigg[\Bigg(\int_\R \Big|\f(t-U_1)\f(t-U_2)\Big| \,dt\Bigg)^m\Bigg] + \E\Bigg[\Bigg(\int_\R \Big|\f(t-U_1)\E_{\pi}(\f(t-U))\Big| \,dt\Bigg)^m\Bigg] \\
& \hspace{26em} + \E\Bigg[\Bigg(\int_\R \big[\E_{\pi}(\f(t-U))\big]^2 \,dt\Bigg)^m\Bigg]\Bigg\},
\end{align*}
with:
\begin{align*}
\E\Bigg[\Bigg(\int_\R \Big|\f(t-U_1)\f(t-U_2)\Big| \,dt\Bigg)^m\Bigg] & = \frac{1}{T^2} \int_{0}^{T}du_1 \int_{0}^{T}du_2 \Bigg(\int_\R \big|\f(t-u_1)\f(t-u_2)\big| \,dt\Bigg)^m \\
& \leq \frac{1}{T^2} \int_{0}^{T}du_1 \int_{u_1-2M}^{u_1+2M}du_2 \left(\int_\R \f^2(x) \,dx\right)^m \\
& \leq \frac{4M}{T} \left(\int_\R \f^2(x) \,dx\right)^m,
\end{align*}
\begin{align*}
\E\Bigg[\Bigg(\int_\R \Big|\f(t-U_1)\E_{\pi}(\f(t-U))\Big| \,dt\Bigg)^m\Bigg] & = \E\Bigg[\Bigg(\int_\R \Big|\f(t-U_1)\left(\frac{1}{T} \int_0^T \f(t-u) \,du\right)\Big| \,dt\Bigg)^m\Bigg] \\
& \leq \frac{1}{T^m} \E\Bigg[\Bigg(\int_0^T \,du \int_\R \big|\f(t-U_1) \f(t-u)\big| \,dt\Bigg)^m\Bigg] \\
& \leq \frac{1}{T^{m+1}} \int_{0}^{T} du_1 \Bigg(\int_0^T du \int_\R \big|\f(t-u_1) \f(t-u)\big| \,dt\Bigg)^m \\
& \leq \frac{1}{T^{m+1}} \int_{0}^{T} du_1 \Bigg(\int_{u_1-2M}^{u_1+2M}du \int_\R \f^2(x) \,dx\Bigg)^m \\
& \leq \frac{(4M)^m}{T^m} \left(\int_\R \f^2(x) \,dx\right)^m
\end{align*}
and
\begin{align*}
\int_\R \big[\E_{\pi}(\f(t-U))\big]^2 \,dt & = \frac{1}{T^2} \int_\R dt \int_{0}^{T}du \, \f(t-u) \int_{0}^{T}dv \, \f(t-v) \\
& = \frac{1}{T^2} \int_{0}^{T}du \int_{0}^{T}dv  \int_\R \f(t-u) \f(t-v) \,dt \\
& \leq \frac{1}{T^2} \int_{0}^{T}du \int_{u-2M}^{u+2M}dv  \int_\R \f^2(x) \,dx \\
& \leq \frac{4M}{T} \int_\R \f^2(x) \,dx.
\end{align*}
So,
\begin{equation}
\E(|\mathcal{W}_3(\f)|^m) \leq K'_3 \ n^{3m/2} \Bigg[\frac{1}{T} \left(\int_\R \f^2(x) \,dx\right)^m + \frac{1}{T^m} \left(\int_\R \f^2(x) \,dx\right)^m\Bigg] \|h\|_{\infty}^m,  \label{A2_PP-2nd:W3}
\end{equation}
with $K'_3$ a positive constant only depending on $m$ and the compact support of $\psi$.
Note that if we had used the same method as for the control of the moments of $W_{2,2}(\f)$, we would not get the correct rate of convergence. We obtain a better rate of convergence thanks to the properties of the biorthogonal wavelet bases used here.

Thus, combining inequalities (\ref{A2_PP-2nd:W0}), (\ref{A2_PP-2nd:W11}), (\ref{A2_PP-2nd:W12}), (\ref{A2_PP-2nd:W13}), (\ref{A2_PP-2nd:W21}), (\ref{A2_PP-2nd:W22}), (\ref{A2_PP-2nd:W23}) and (\ref{A2_PP-2nd:W3}) yields
\begin{align*}
& \E(V(\f)^m) \\
& \leq K \ \Bigg\{ n^m \left(\int_\R \f^2(x) \,dx\right)^m \|h\|_{\infty}^m + \frac{n^{2m}}{T^m} \left(\int_\R \f^2(x) \,dx\right)^m \left(\int_\R h(x) \,dx\right)^m \\
& \hspace{5em} \quad + \quad \frac{n^{3m/2}}{T^m} \left(\int_\R |\f(x)| \,dx\right)^{2m} \|h\|_{\infty}^m + \frac{n^{3m/2}}{T^m} \left(\int_\R \f^2(x) \,dx\right)^m \left(\int_\R h(x) \,dx\right)^m \\
& \hspace{5em} \qquad + \frac{n^{3m/2}}{T^m} \left(\int_\R \f^2(x) \,dx\right)^m \left(\int_\R h(x) \,dx\right)^m \quad + \quad n^m \left(\int_\R \f^2(x) \,dx\right)^{m} \|h\|_{\infty}^m \\
& \hspace{5em} \qquad + \frac{n^{2m}}{T^m} \left(\int_\R \f^2(x) \,dx\right)^{m} \left(\int_\R h(x) \,dx\right)^{m} + n^m \left(\int_\R \f^2(x) \,dx\right)^{m} \|h\|_{\infty}^m \\
& \hspace{5em} \quad + \quad n^{3m/2} \Bigg[\frac{1}{T} \left(\int_\R \f^2(x) \,dx\right)^m + \frac{1}{T^m} \left(\int_\R \f^2(x) \,dx\right)^m\Bigg] \|h\|_{\infty}^m \Bigg\} \\
& \leq K \ \Bigg\{ \frac{n^{3m/2}}{T^{m}} \|\f\|_1^{2m} \|h\|_{\infty}^m + \frac{n^{2m}}{T^{m}} \|\f\|_2^{2m} \|h\|_{1}^m + \left[n^m + \frac{n^{3m/2}}{T}\right] \|\f\|_2^{2m} \|h\|_{\infty}^m \Bigg\},
\end{align*}
with $K$ a positive constant only depending on $m$ and the compact support of $\psi$.
So, we obtain
\begin{equation}
\E(V(\f)^p) \leq K \ \Bigg\{ \frac{n^{3p/2}}{T^{p}} \|\f\|_1^{2p} \|h\|_{\infty}^{p} + \frac{n^{2p}}{T^{p}} \|\f\|_2^{2p} \|h\|_{1}^{p} + \left[n^{p} + \frac{n^{3p/2}}{T}\right] \|\f\|_2^{2p} \|h\|_{\infty}^{p} \Bigg\},  \label{A2_PP-2nd}
\end{equation}
with $K$ a positive constant only depending on $p$ and the compact support of $\psi$.

\medskip

To conclude for the first term of (\ref{A2_general}), using inequalities (\ref{A2_PP-1st}) and (\ref{A2_PP-2nd}) in (\ref{A2_PP}), we have
\begin{equation}
\begin{split}
& \E(|G(\f) - \E(G(\f)|U_1,\ldots,U_n)|^{2p}) \\
& \leq K \ \Bigg\{ \frac{n^2}{T^{2p}} \|\f\|_1^{2p} \|h\|_1 + \frac{n^{p+1}}{T^p} \|\f\|_2^{2p} \|h\|_1 + \frac{n^2}{T} \|\f\|_2^2 \|\f\|_{\infty}^{2p-2} \|h\|_1 + n \|\f\|_2^2 \|\f\|_{\infty}^{2p-2} \|h\|_{\infty} \\
& \hspace{4em} + \frac{n^{3p/2}}{T^{p}} \|\f\|_1^{2p} \|h\|_{\infty}^{p} + \frac{n^{2p}}{T^{p}} \|\f\|_2^{2p} \|h\|_{1}^{p} + \left[n^{p} + \frac{n^{3p/2}}{T}\right] \|\f\|_2^{2p} \|h\|_{\infty}^{p} \Bigg\}.
\end{split}  \label{A2_PP-end}
\end{equation}

\bigskip

Now, we have to focus on the second term of (\ref{A2_general}). Recall the definition (\ref{def_W}) of $W(\f)$ \[W(\f)=W_1(\f)+W_2(\f),\]
with \[W_1(\f) = (n-1) \sum_{i=1}^{n} \int_\R \big[\f(t-U_i)-\E_{\pi}(\f(t-U))\big] \E_{\pi}(h(t-U)) \,dt,\]
and \[W_2(\f) = \sum_{1 \leq i \neq j \leq n} g(U_i,U_j),\]
where \[g(U_i,U_j) = \int_\R \big[\f(t-U_i)-\E_{\pi}(\f(t-U))\big] \big[h(t-U_j)-\E_{\pi}(h(t-U))\big] \,dt.\]
So, \[\E(|W(\f)|^{2p}) \leq 2^{2p-1} \big[\E(|W_1(\f)|^{2p}) + \E(|W_2(\f)|^{2p})\big].\]

\medskip

On the one hand, we have to control $\E(|W_1(\f)|^{2p})$.
We use Rosenthal's inequality: there exists a positive constant $C(p)$ only depending on $p$ such that
\begin{align*}
& \E(|W_1(\f)|^{2p}) \\
& = (n-1)^{2p} \E\left(\left|\sum_{i=1}^{n} \int_\R \big[\f(t-U_i)-\E_{\pi}(\f(t-U))\big] \E_{\pi}(h(t-U)) \,dt\right|^{2p}\right) \\
& \leq (n-1)^{2p} \times C(p) \Bigg(n\E\left(\left|\int_\R \big[\f(t-U_1)-\E_{\pi}(\f(t-U))\big] \E_{\pi}(h(t-U)) \,dt\right|^{2p}\right) \\
& \hspace{10em} + n^{p}\left[\var\left(\int_\R \big[\f(t-U_1)-\E_{\pi}(\f(t-U))\big] \E_{\pi}(h(t-U)) \,dt\right)\right]^{p}\Bigg).
\end{align*}
But, applying \ref{lemINEQ},
\begin{align*}
& \E\left(\left|\int_\R \big[\f(t-U_1)-\E_{\pi}(\f(t-U))\big] \E_{\pi}(h(t-U)) \,dt\right|^{2p}\right) \\
& \leq \frac{1}{T^{2p}} \E\left(\left|\int_\R \big[\f(t-U_1)-\E_{\pi}(\f(t-U))\big] \,dt\right|^{2p}\right) \left(\int_\R h(x) \,dx\right)^{2p} \\
& \leq \frac{2^{2p}}{T^{2p}} \left(\int_\R |\f(x)| \,dx\right)^{2p} \left(\int_\R h(x) \,dx\right)^{2p}
\end{align*}
and
\begin{align*}
& \var\left(\int_\R \big[\f(t-U_1)-\E_{\pi}(\f(t-U))\big] \E_{\pi}(h(t-U)) \,dt\right) \\
& = \var\left(\int_\R \f(t-U_1) \E_{\pi}(h(t-U)) \,dt\right) \\
& \leq \E \left[\left(\int_\R |\f(t-U_1)| \E_{\pi}(h(t-U)) \,dt\right)^2\right] \\
& \leq \frac{1}{T^2} \E \left[\left(\int_\R |\f(t-U_1)| \,dt\right)^2\right] \left(\int_\R h(x) \,dx\right)^{2} \\
& \leq \frac{1}{T^2} \left(\int_\R |\f(x)| \,dx\right)^2 \left(\int_\R h(x) \,dx\right)^{2}.
\end{align*}
So,
\begin{equation}
\E(|W_1(\f)|^{2p}) \leq K_1 \ n^{2p} \left[\frac{n}{T^{2p}} \left(\int_\R |\f(x)| \,dx\right)^{2p} + \frac{n^{p}}{T^{2p}} \left(\int_\R |\f(x)| \,dx\right)^{2p} \right] \left(\int_\R h(x) \,dx\right)^{2p},  \label{A2_W1}
\end{equation}
with $K_1$ a positive constant only depending on $p$.

\medskip

And on the other hand, we have to control $\E(|W_2(\f)|^{2p})$.
We have: \[W_2(\f) = \sum_{1 \leq i \neq j \leq n} g(U_i,U_j).\] We use Theorem 3.3 of \cite{GLZ} associated with Theorem 1 of \cite{dPMS} (we keep the same notations of \cite{GLZ}). We set $h_{i,j}=\left\{\begin{array}{cl} 0 & \mbox{if $i=j$} \\ g & \mbox{otherwise}\end{array}\right.$ and we consider $(U^{(1)}_i, i=1 \ldots n)$ and $(U^{(2)}_i, i=1 \ldots n)$ two independent copies of $(U_i, i=1 \ldots n)$. With Theorem 3.3 of \cite{GLZ}, there exists an universal constant $K$ such that
\[\E\left(\left|\sum_{1 \leq i , j \leq n} h_{i,j}(U^{(1)}_i,U^{(2)}_j)\right|^{2p}\right) \leq K^{2p} \big[(2p)^p C^{2p} + (2p)^{2p} D^{2p} + (2p)^{3p} B^{2p} + (2p)^{4p} A^{2p}\big],\]
where
\begin{itemize}
  \item $\displaystyle A = \max_{i,j} \|h_{i,j}\|_{\infty} = \|g\|_{\infty}$. But, for all $(x,y) \in \R^2$,
      \begin{align*}
      |g(x,y)| & = \left|\int_\R \big[\f(t-x)-\E_{\pi}(\f(t-U))\big] \big[h(t-y)-\E_{\pi}(h(t-U))\big] \,dt\right| \\
      & \leq 4 \int_\R |\f(x)|\,dx \|h\|_{\infty},
      \end{align*}
      using equality (b) of \ref{lemINEQ} with $f=\f$. So,
      \begin{equation}
      A \leq 4 \int_\R |\f(x)|\,dx \|h\|_{\infty},  \label{A2_W2-A}
      \end{equation}
  \item $\displaystyle C^2 = \sum_{i,j} \E(h_{i,j}^2(U^{(1)}_i,U^{(2)}_j)) = \sum_{i \neq j} \E(g^2(U_i,U_j))$. But, for all $i \neq j$,
      \begin{align*}
      & \E(g^2(U_i,U_j)) \\
      & \leq \E \left[\left(\int_\R \big[\f(t-U_i)-\E_{\pi}(\f(t-U))\big] \big[h(t-U_j)-\E_{\pi}(h(t-U))\big]
      \,dt\right)^2\right] \\
      & \leq \E\left[\int_\R [\f(t-U_i)-\E_{\pi}(\f(t-U))]^2 |h(t-U_j)-\E_{\pi}(h(t-U))| \,dt \int_\R
      |h(t-U_j)-\E_{\pi}(h(t-U))| \,dt\right] \\
      & \leq 2 \E\left[\int_\R \big[\f(t-U_i)-\E_{\pi}(\f(t-U))\big]^2 \big|h(t-U_j)-\E_{\pi}(h(t-U))\big| \,dt\right] \int_\R h(x) \,dx \\
      & \leq \frac{4}{T} \int_\R \E\left(\big[\f(t-U_i)-\E_{\pi}(\f(t-U))\big]^2\right) \,dt \left(\int_\R h(x) \,dx\right)^2 \\
      & \leq \frac{4}{T} \int_\R \f^2(x) \,dx \left(\int_\R h(x) \,dx\right)^2,
      \end{align*}
      using \ref{lemINEQ}.
      So,
      \begin{equation}
      C^2 \leq \frac{4n(n-1)}{T} \int_\R \f^2(x) \,dx \left(\int_\R h(x) \,dx\right)^2,  \label{A2_W2-C}
      \end{equation}
  \item $\displaystyle B^2 = \max\left[\bigg\|\sum_i \E_1 (h_{i,j}^2(U_i,y))\bigg\|_{\infty},\bigg\|\sum_j \E_2 (h_{i,j}^2(x,U_j))\bigg\|_{\infty}\right]$, with
      \[\E_1(h_{i,j}^2(U_i,y)) = \left\{\begin{array}{cl} \E_{\pi}(g^2(U,y)) & \mbox{if $i \neq j$} \\ 0 & \mbox{otherwise}\end{array}\right.
      \leq \frac{4}{T} \int_\R \f^2(x) \,dx \left(\int_\R h(x) \,dx\right)^2\]
      and
      \[\E_1(h_{i,j}^2(x,U_j)) = \left\{\begin{array}{cl} \E_{\pi}(g^2(x,U)) & \mbox{if $i \neq j$} \\ 0 & \mbox{otherwise}\end{array}\right.
      \leq \frac{16}{T} \int_\R \f^2(x) \,dx \left(\int_\R h(x) \,dx\right)^2,\]
      using established inequalities to get equation (\ref{A1_W2-B}) in the proof of the assumption (A1). So,
      \begin{equation}
      B^2 \leq \frac{16n}{T} \int_\R \f^2(x) \,dx \left(\int_\R h(x) \,dx\right)^2, \label{A2_W2-B}
      \end{equation}
  \item $\displaystyle D = \sup\left\{\E\bigg(\sum_{i,j} h_{i,j}(U^{(1)}_i,U^{(2)}_j)a_i(U^{(1)}_i)b_j(U^{(2)}_j)\bigg) : \E\bigg(\sum_i a_i^2(U_i)\bigg) \leq 1 , \E\bigg(\sum_j b_j^2(U_j)\bigg) \leq 1\right\}$. By using established inequalities to get equation (\ref{A1_W2-D}) in the proof of the assumption (A1), we obtain
      \begin{equation}
      D \leq 2n \sqrt{\frac{1}{T} \int_\R \f^2(x) \,dx} \int_\R h(x) \,dx.  \label{A2_W2-D}
      \end{equation}
\end{itemize}

Moreover, we use the equivalence of Theorem 3.3 of \cite{GLZ} and the decoupling inequality provided in Theorem 1 of \cite{dPMS} to obtain the following upper bound of $\E(|W_2(\f)|^{2p})$:
\begin{equation}
\begin{split}
\E(|W_2(\f)|^{2p}) & \leq K_2 \ \Bigg[\frac{n^{2p}}{T^p} \left(\int_\R \f^2(x) \,dx\right)^p \left(\int_\R h(x) \,dx\right)^{2p} + \frac{n^p}{T^p} \left(\int_\R \f^2(x) \,dx\right)^{p} \left(\int_\R h(x) \,dx\right)^{2p} \\
& \hspace{25em} + \left(\int_\R |\f(x)| \,dx\right)^{2p} \|h\|_{\infty}^{2p}\Bigg],
\end{split}  \label{A2_W2}
\end{equation}
with $K_2$ a positive constant only depending on $p$.

\bigskip

Finally, by using inequalities (\ref{A2_PP-end}), (\ref{A2_W1}) and (\ref{A2_W2}) in (\ref{A2_general}), we obtain:
\begin{align*}
& \E(|\hat{\beta}_{\lambda} - \beta_{\lambda}|^{2p}) \\
& \leq K \ \Bigg\{ \frac{1}{n^{2p-2}T^{2p}} \|\f\|_1^{2p} \|h\|_1 + \frac{1}{n^{p-1}T^p} \|\f\|_2^{2p} \|h\|_1 + \frac{1}{n^{2p-2}T} \|\f\|_2^2 \|\f\|_{\infty}^{2p-2} \|h\|_1 \\
& \hspace{3em} + \frac{1}{n^{2p-1}} \|\f\|_2^2 \|\f\|_{\infty}^{2p-2} \|h\|_{\infty} + \frac{1}{n^{p/2}T^{p}} \|\f\|_1^{2p} \|h\|_{\infty}^{p} + \frac{1}{T^{p}} \|\f\|_2^{2p} \|h\|_{1}^{p} \\
& \hspace{3em} + \frac{1}{n^{p}} \|\f\|_2^{2p} \|h\|_{\infty}^{p} + \frac{1}{n^{p/2}T} \|\f\|_2^{2p} \|h\|_{\infty}^{p} \quad + \quad  \frac{n}{T^{2p}} \|\f\|_1^{2p} \|h\|_{1}^{2p} + \frac{n^{p}}{T^{2p}} \|\f\|_1^{2p} \|h\|_{1}^{2p} \\
& \hspace{3em} \quad + \quad \frac{1}{T^p} \|\f\|_2^{2p} \|h\|_{1}^{2p} + \frac{1}{n^pT^p}  \|\f\|_2^{2p} \|h\|_{1}^{2p} + \frac{1}{n^{2p}} \|\f\|_1^{2p} \|h\|_{\infty}^{2p} \Bigg\}
\end{align*}
and so,
\begin{align*}
& \left[\E\left(|\hat{\beta}_{\lambda}-\beta_{\lambda}|^{2p}\right)\right]^{\frac{1}{p}} \\
& \leq K \ \Bigg\{ \frac{1}{n^{2-2/p}T^{2}} \|\f\|_1^{2} \|h\|_1^{1/p} + \frac{n}{T^{2}} \|\f\|_1^{2} \|h\|_{1}^{2} + \frac{1}{n^{1/2}T} \|\f\|_1^{2} \|h\|_{\infty} + \frac{1}{n^{2}} \|\f\|_1^{2} \|h\|_{\infty}^{2} \\
& \hspace{3em} + \frac{1}{n^{1-1/p}T} \|\f\|_2^{2} \|h\|_1^{1/p} + \frac{1}{T} \|\f\|_2^{2} \|h\|_{1} + \frac{1}{T} \|\f\|_2^{2} \|h\|_{1}^{2} + \left[\frac{1}{n} + \frac{1}{n^{1/2}T^{1/p}}\right] \|\f\|_2^{2} \|h\|_{\infty} \\
& \hspace{3em} + \left[\frac{1}{n^{2-2/p}T^{1/p}} + \frac{1}{n^{2-1/p}}\right] \|\f\|_2^{2/p} \|\f\|_{\infty}^{2-2/p} \|h\|_{\infty}^{1/p} \Bigg\},
\end{align*}
with $K$ a positive constant only depending on $p$ and the compact support of $\psi$.

Recall that for any $\lambda=(j,k)\in\Lambda$, we have: \[\|\f\|_1 \leq 2^{-j/2}M_{\psi,1}, \quad \|\f\|_2 \leq M_{\psi,2} \quad \mbox{and} \quad \|\f\|_\infty \leq 2^{j/2}M_{\psi,\infty}\]
We consider $1<p<\infty$ and we fix $1<q<\infty$ such that $\frac{1}{p} + \frac{1}{q} = 1$, so that
\begin{align*}
\left[\E\left(|\hat{\beta}_{\lambda}-\beta_{\lambda}|^{2p}\right)\right]^{\frac{1}{p}} & \leq K \ \Bigg\{ \frac{n}{T^{2}} \|h\|_{1}^{2} + \frac{1}{n^{2}} \|h\|_{\infty}^{2} + \frac{1}{n^{1/q}T} \|h\|_1^{1/p} + \frac{1}{T} \|h\|_{1} + \frac{1}{T} \|h\|_{1}^{2} \\
& \hspace{3em} + \left[\frac{1}{n} + \frac{1}{n^{1/2}T^{1/p}}\right] \|h\|_{\infty} + \left[\frac{1}{n^{2/q}T^{1/p}} + \frac{1}{n^{1+1/q}}\right] 2^{j_0/q} \|h\|_{\infty}^{1/p} \Bigg\},
\end{align*}
with $K$ a positive constant depending on $p$, $\|\psi\|_1$, $\|\psi\|_2$, $\|\psi\|_{\infty}$ and the compact support of $\psi$.

\medskip

Finally, choosing $p=2$, Assumption (A2) is fulfilled with
\[R= C_R \ \Bigg\{ \frac{1}{n}  + \frac{2^{j_0/2}}{n^{3/2}} + \frac{2^{j_0/2}}{nT^{1/2}} + \frac{n}{T^{2}}  \Bigg\},\]
where $C_R$ is a positive constant depending on $\|h\|_1$, $\|h\|_{\infty}$, $\|\psi\|_1$, $\|\psi\|_2$, $\|\psi\|_{\infty}$ and the compact support of $\psi$,
\[H_\lambda = \1_{\lambda \in \Gamma}\]
and $\varepsilon=1$.

\subsubsection{Proof of Assumption (A3)}

To shorten mathematical expressions, we denote $\eta_\lambda = \eta_\lambda(\gamma,\Delta)$ in the sequel. The following inequality:
\[\P\left( |\hat{\beta}_{\lambda}-\beta_{\lambda}| > \kappa \eta_{\lambda} , |\hat{\beta}_{\lambda}| > \eta_{\lambda} \right) \leq H_{\lambda}\zeta\]
is obvious with $\zeta = \omega$, which proves Assumption (A3) choosing $\theta = \frac{1+\varepsilon}{\varepsilon}$.

\subsubsection{Completion of the proof of \ref{thm3}}

Therefore we can apply \ref{thmRBR}: the estimator $\tilde{\beta} = \left(\hat{\beta}_{\lambda} \1_{|\hat{\beta}_{\lambda}| \geq \eta_{\lambda}} \1_{\lambda \in \Gamma}\right)_{\lambda \in \Lambda}$ satisfies
\begin{align*}
\frac{1-\kappa^{2}}{1+\kappa^{2}} \E\left(\|\tilde{\beta}-\beta\|_{\ell_{2}}^{2}\right) & \leq \E\left(\inf_{m \subset \Gamma} \left\{\frac{1+\kappa^{2}}{1-\kappa^{2}} \sum_{\lambda \not \in m} \beta_{\lambda}^{2} + \frac{1-\kappa^{2}}{\kappa^{2}} \sum_{\lambda \in m}(\hat{\beta}_{\lambda}-\beta_{\lambda})^{2} + \sum_{\lambda \in m} \eta_{\lambda}^{2}\right\}\right) + L D \sum_{\lambda \in \Gamma} H_{\lambda} \\
& \leq \inf_{m \subset \Gamma} \left\{\frac{1+\kappa^{2}}{1-\kappa^{2}} \sum_{\lambda \not \in m} \beta_{\lambda}^{2} + \frac{1-\kappa^{2}}{\kappa^{2}} \sum_{\lambda \in m}\E((\hat{\beta}_{\lambda}-\beta_{\lambda})^{2}) + \sum_{\lambda \in m} \E(\eta_{\lambda}^{2})\right\} + L D \sum_{\lambda \in \Gamma} H_{\lambda},
\end{align*}
with
\begin{itemize}
  \item for all $\lambda = (j,k)$ in $\Gamma$, \[\E((\hat{\beta}_{\lambda}-\beta_{\lambda})^{2}) = \var(\hat{\beta}_{\lambda}) \leq K \ \left\{\frac{1}{n} + \frac{1}{T} + \frac{2^{-j}n}{T^2}\right\},\] where $K$ is a positive constant depending on $\|h\|_1$, $\|h\|_{\infty}$, $\|\psi\|_1$ and $\|\psi\|_2$ (see \ref{lemEST});
  \item for all $\lambda = (j,k)$ in $\Gamma$,
      \[\eta_{\lambda} \leq K \left(\sqrt{j_0 \widetilde{V}\left(\frac{\f}{n}\right)} + j_0 B\left(\frac{\f}{n}\right) + \tilde{\Delta} \frac{N_\R}{n}\right)\] where $K$ depends on $\varepsilon$, $\kappa$, $\gamma$, $\|\psi\|_1$, $\|\psi\|_2$ and $\|\psi\|_{\infty}$ and \[\tilde{\Delta} = \frac{j_0^2 2^{j_0/2}}{n} + \frac{j_0}{\sqrt{T}} +  \frac{\sqrt{j_0n}}{T};\]
  \item $L D = \frac{R}{\kappa^{2}} \big((1+\theta^{-1/2}) \omega^{1/2} + (1+\theta^{1/2}) \varepsilon^{1/2} \zeta^{1/2}\big) \leq K \ R \big(e^{-\kappa^2j_0\gamma/2} + \exp{(-g(\varepsilon_0) n\|h\|_1/2)}\big)$, where $K$ is a positive constant depending only on $\varepsilon$ and $\kappa$,
  \item $\displaystyle \sum_{\lambda \in \Gamma} H_{\lambda} = |\Gamma|$, where $|\Gamma|$ is the cardinal of the set $\Gamma$. So, we can upper bound this quantity by $K \ 2^{j_0}$, where $K$ is a positive constant depending only on the compact support of $h$ and the compact support of $\psi$.
\end{itemize}
Recall that $\varepsilon = 1 $, $\kappa\in]0;1[$ will be fixed in the sequel and $\gamma > 0$, according Assumption (A1).

\medskip

It remains to compute $\E(\eta_{\lambda}^{2})$. Let $\lambda\in\Gamma$. We have:
\[\E(\eta_{\lambda}^{2}) \leq K \ \left(\frac{j_0}{n^2} \E(\widetilde{V}(\f)) + \frac{j_0^2}{n^2} \E(B^2(\f)) + \Bigg\{ \frac{j_0n}{T^2} + \frac{j_0^2}{T} +  \frac{2^{j_0}j_0^4}{n^2} \Bigg\} \frac{\E(N_\R^2)}{n^2}\right),\]
with $K$ depending on $\varepsilon$, $\kappa$, $\gamma$, $\|\psi\|_1$, $\|\psi\|_2$ and $\|\psi\|_{\infty}$  and $\E(N_\R^2) = n\|h\|_1 + n^2\|h\|_1^2 \leq 2n^2 M_{h,1}^2 $. \\
We control $\widetilde{V}(\f)$ in expectation and we recall that $\alpha=\kappa^2j_0\gamma$.
\begin{equation}
\E(\widetilde{V}(\f)) \leq \E(\hat{V}(\f)) + \sqrt{2\alpha\E(\hat{V}(\f))\E(B^2(\f))} + 3\alpha \E(B^2(\f)),  \label{completion_V}
\end{equation}
with using inequality (\ref{pr-lemEST_PP}),
\[\E(\hat{V}(\f)) = \E(V(\f)) \leq K \ \left\{n \|h\|_{\infty} + \frac{n^2}{T} \|h\|_1\right\},\]
where $K$ is a positive constant depending only on $\|\psi\|_2$.

Now, we focus on $\E(B^2(\f))$ where
\begin{align*}
B(\f) & = B(\f) = \left\|\sum_{i=1}^{n} \left[\f(\cdot-U_i) - \frac{n-1}{n} \E_{\pi}(\f(\cdot-U))\right]\right\|_{\infty} \\
& = \sup_{t \in \R} \left|\sum_{i=1}^{n} \left[\f(t-U_i) - \frac{n-1}{n} \E_{\pi}(\f(t-U))\right]\right| \\
& \leq \tilde{B}(\f) + \frac{1}{T} \|\f\|_1,
\end{align*}
with $\displaystyle \tilde{B}(\f) = \sup_{t \in \R} \left|\sum_{i=1}^{n} \big[\f(t-U_i) - \E_{\pi}(\f(t-U))\big]\right|$. \\
Then, we have to control $\E(\tilde{B}(\f))$. Since it is a decomposition biorthogonal wavelet, $\f$ is a piecewise constant function and we can write:
\[\displaystyle \f = \sum_{l=1}^{N} c_l \1_{[a_l;b_l]},\]
where $N\in\N^*$ and for any $l\in \{1,\ldots,N\}$, $a_l, b_l, c_l \in \R$ and $a_l<b_l$.
It is easy to see that
\[\displaystyle \tilde{B}(\f) \leq \sum_{l=1}^{N} \tilde{B}(c_l \1_{[a_l;b_l]}) = \sum_{l=1}^{N} |c_l|\tilde{B}(\1_{[a_l;b_l]}).\]
It remains to compute $\E(\tilde{B}(\1_{[a;b]}))$ for some interval $[a;b]$.
\begin{align*}
\tilde{B}(\1_{[a;b]}) & = \sup_{t \in \R} \left|\sum_{i=1}^{n} \big[\1_{[a;b]}(t-U_i) - \E_{\pi}(\1_{[a;b]}(t-U))\big]\right| \\
& \leq \sup_{B_t, t \in \R} \left|\sum_{i=1}^{n} \big[\1_{B_t}(U_i) - \E_{\pi}(\1_{B_t}(U))\big]\right|,
\end{align*}
where for any $t \in \R$, $B_t = [t-b;t-a]$. \\
We set $\mathcal{B}=\{B_t, t \in \R\}$ and for every integer $n$, $\displaystyle m_n(\mathcal{B}) = \sup_{A \subset \R, |A|=n} |\{A \cap B_t, t \in \R\}|$. It is easy to see that \[m_n(\mathcal{B}) \leq 1 + \frac{n(n+1)}{2}\] and so, the VC-dimension of $\mathcal{B}$ defined by $sup \{n \geq 0, m_n(\mathcal{B}) = 2^n\}$ is bounded by 2 (see Definition 6.2 of \cite{Mas}). \\
By applying Lemma 6.4 of \cite{Mas}, we obtain:
\[\sqrt{n}\E(\tilde{B}(\1_{[a;b]})) \leq \frac{K}{2}\sqrt{2},\]
where $K$ is an absolute constant. So, for any $\lambda$ in $\Gamma$,
\[\E(\tilde{B}(\f)) \leq \frac{K}{\sqrt{n}}.\]
But, we want an upper bound of $\E(\tilde{B}^2(\f))$. For this, we use Theorem 11 of \cite{BBLM}:
\[\big[\E(\tilde{B}^2(\f))\big]^{1/2} \leq K \, \bigg\{\E(\tilde{B}(\f))+\|\mathcal{M}\|_2\bigg\},\]
where \[\mathcal{M} = \max_{1 \leq i \leq n} \sup_{t \in \R} \big|\f(t-U_i) - \E_{\pi}(\f(t-U))\big|.\]
Hence, \[\|\mathcal{M}\|_2^2 \leq 4 \|\f\|_{\infty}^2 \leq K \ 2^{j},\] with $K$ a constant only depending on $\|\psi\|_{\infty}$.

Finally,
\begin{align}
\E(B^2(\f)) & \leq K \ \left\{\E(\tilde{B}^2(\f)) + \frac{2^{-j}}{T^2}\right\} \nonumber\\
& \leq K \ \left\{\big[\E(\tilde{B}(\f))\big]^2 + 2^{j} + \frac{2^{-j}}{T^2}\right\} \nonumber\\
& \leq K \ \left\{\frac{1}{n} + 2^{j} + \frac{2^{-j}}{T^2}\right\},  \label{completion_B}
\end{align}
with $K$ a constant only depending on $\|\psi\|_1$ and $\|\psi\|_{\infty}$.

Then combining (\ref{completion_V}) and (\ref{completion_B}) yields
\[\E(\eta_{\lambda}^{2})  \leq K \ \Bigg\{\frac{j_0}{n} + \frac{j_0^{3/2}2^{j_0/2}}{n^{3/2}} + \frac{j_0^2 2^{j_0}}{n^2} + \frac{j_0^4 2^{j_0}}{n^2} + \frac{j_0}{T}  + \frac{j_0^2}{T} + \frac{j_0^{3/2}2^{j_0/2}}{nT^{1/2}} + \frac{j_0 n}{T^2}  \Bigg\},\]
where $K$ is a constant depending on $\gamma$, $\|h\|_1$, $\|h\|_{\infty}$, $\|\psi\|_1$, $\|\psi\|_2$ and $\|\psi\|_{\infty}$,
which concludes the proof of \ref{thm3} by setting
\[F(j_0,n,T)= \frac{j_0}{n} + \frac{j_0^{3/2}2^{j_0/2}}{n^{3/2}} + \frac{j_0^4 2^{j_0}}{n^2} + \frac{j_0^2}{T} + \frac{j_0^{3/2}2^{j_0/2}}{nT^{1/2}} + \frac{j_0 n}{T^2}.\]

\end{document}